\documentclass[journal]{IEEEtran}

\usepackage{subcaption}
\usepackage{mathtools}
\usepackage[hidelinks]{hyperref}
\usepackage{algorithmic}
\usepackage[ruled,vlined]{algorithm2e}       
\usepackage{amsthm,amsmath,amssymb}
\allowdisplaybreaks[4]
\usepackage{color}
\usepackage{cite}
\usepackage{booktabs}
\usepackage{makecell}
\usepackage{multirow}

\newtheorem{lemma}{Lemma}
\newtheorem{corollary}{Corollary}
\newtheorem{theorem}{Theorem}

\newtheorem{remark}{Remark}

\newtheorem{example}{Example}

\begin{document}



\title{Policy Evaluation in Distributional LQR}

\author{Zifan Wang, Yulong Gao, Siyi Wang, Michael M. Zavlanos, Alessandro Abate, and Karl H. Johansson
\thanks{* This work was supported in part by Swedish Research Council Distinguished Professor Grant 2017-01078, Swedish Research Council International Postdoc Grant 2021-06727, Knut and Alice Wallenberg Foundation Wallenberg Scholar Grant, Swedish Strategic Research Foundation SUCCESS Grant FUS21-0026, AFOSR under award \#FA9550-19-1-0169, and  NSF under award CNS-1932011.} 
\thanks{Zifan Wang and Karl H. Johansson are with Division of Decision and Control Systems, School of Electrical Enginnering and Computer Science, KTH Royal Institute of Technology, and also with Digital Futures, SE-10044 Stockholm, Sweden. Email: \{zifanw,kallej\}@kth.se.}
\thanks{Yulong Gao and Alessandro Abate are with Department of Computer Science, University of Oxford, OX13QD Oxford, U.K.  Email: \{ yulong.gao, alessandro.abate\}@cs.ox.ac.uk} 
\thanks{Siyi Wang is with the School of Computation, Information and Technology, Technical University of Munich, 80333 Munich, Germany. Email: siyi.wang@tum.de}
\thanks{Michael M. Zavlanos is with the Department of Mechanical Engineering and Materials Science, Duke University, Durham, NC, USA. Email: michael.zavlanos@duke.edu}
}

\maketitle
\begin{abstract}%
Distributional reinforcement learning (DRL) enhances the understanding of the effects of the randomness  in the environment by letting agents learn the distribution of a random return, rather than its expected value as in standard reinforcement learning.  
Meanwhile, a challenge in DRL is that the policy evaluation typically relies on the representation of the return distribution, which needs to be carefully designed.
In this paper, we address this challenge for the special class of DRL problems that rely on a discounted linear quadratic regulator (LQR), which we call \emph{distributional LQR}. 
Specifically, we provide  a closed-form expression for the distribution of the random return, which is applicable for all types of exogenous disturbance as long as it is independent and identically distributed (i.i.d.).
We show that the variance of the random return is bounded if the fourth moment of the exogenous disturbance is bounded. Furthermore, we investigate the sensitivity of the return distribution to model perturbations.
While the proposed exact return distribution consists of infinitely many random variables, we show that this distribution can be well approximated by a finite number of random variables. The associated approximation error can be analytically bounded under mild assumptions. 
When the model is unknown, we propose a model-free approach for estimating the return distribution, supported by sample complexity guarantees.
Finally, we extend our approach to partially observable linear systems.
Numerical experiments are provided to illustrate the theoretical results.
\end{abstract}

\begin{IEEEkeywords}
  Distributional LQR, distributional RL, distribution sensitivity, policy evaluation, partially observable system
\end{IEEEkeywords}
\IEEEpeerreviewmaketitle

\section{Introduction}

In reinforcement learning (RL), the value of implementing a policy at a given state is captured by a value function, which models the expected sum of returns following this prescribed policy. 
Recently, \cite{bellemare2017distributional}  proposed the notion of distributional reinforcement learning (DRL), which learns the return distribution of a policy from a given state, instead of only its expected return. 
Compared to the scalar expected value function, the return distribution is infinite-dimensional and contains far more information. 
It is, therefore, not surprising that a few DRL algorithms, including C51 \cite{bellemare2017distributional}, D4PG \cite{barth2018distributed}, QR-DQN \cite{dabney2018distributional} and SDPG \cite{singh2022sample}, dramatically improve the empirical performance in practical applications over their non-distributional counterpart. 
By encompassing the entire distribution, DRL is able to provide a comprehensive framework, for instance, for risk-averse learning, facilitating a deeper understanding and more effective management of uncertainties \cite{kamran2021minimizing,zhang2022safe,liang2022bridging,min2019deep}.

In parallel with the celebrated Bellman equation in the traditional RL, an alternative random variable (or distributional) Bellman equation acts as the theoretical foundation of DRL.  It has been shown in \cite{bellemare2017distributional}
that the return distribution satisfies the distributional Bellman equation and the distributional Bellman operator is a contraction in (the maximum form of) the Wasserstein metric between probability distributions. A natural yet fundamental question in DRL is:
\begin{center}
   \emph{Given a policy, how to (exactly) characterise the random return that fulfils the random variable Bellman equation?}
\end{center}
The answer to this question provides the structural information of the return distribution, which enables a better understanding of the value of implementing a policy in the DRL setting. 

To the best of our knowledge, this problem has received limited attention.  One of the challenges is the computational intractability arising from the fact that the return distribution is in an infinite-dimensional space.  Approximations thus become necessary for practical implementation - cf. categorical \cite{bellemare2017distributional}, quantile function 
\cite{dabney2018distributional}, and sample-based \cite{singh2022sample} methods. Furthermore, although some recent efforts have been devoted to  applying DRL to partially observable systems \cite{xu2023decision}, no theoretical foundations, including the characterisation of the random return, have been built for these partially observable models.

In this paper, we solve the above problem for discrete-time linear systems with stochastic additive disturbances. More specifically, we characterise the random cost for the classical discounted linear quadratic regulator (LQR) problem, which we term \emph{distributional LQR}. We investigate the fundamental properties of the characterized random cost.  Furthermore, we explore the extension to  partially observable systems and derive fundamental properties of the characterized random cost.

\subsection{Related Work} 
The problem under investigation falls within the domain of policy evaluation in DRL, specifically focusing on predicting the full probability distribution.
This task poses a unique challenge because the full probability distribution is infinitely dimensional, necessitating the use of distribution parametrization techniques to render it computationally feasible.
Bellemare \textit{et al}. \cite{bellemare2017distributional} propose a categorical method that discretizes the return distribution by partitioning the return distribution into a finite number of uniformly spaced atoms in a fixed region. Subsequent work \cite{rowland2018analysis} delves into the convergence analysis of categorical policy evaluation and shows that the distributional projected Bellman operator with categorical representation is a contraction with respect to the Cram\'{e}r distance metric.
One drawback of the categorical representation is that it relies on prior knowledge of the range of the returned values. 
To address this limitation, \cite{dabney2018distributional} proposes a quantile temporal-difference learning algorithm that learns the quantiles of a probability distribution, and its convergence property is established in \cite{rowland2023analysis} using the Wasserstein-$\infty$ metric.
However, most of the existing algorithms and analysis of DRL are tailored to address problems with discrete state spaces, which cannot be applied to the linear quadratic control problem with continuous state space.
It is worth mentioning that the works \cite{singh2020improving,singh2022sample} investigate DRL with continuous state space and use a reparameterization method to represent the distribution of random variables through a neural network. 
Despite these significant advancements, there is no theoretical guarantee regarding the quality of the learned distributions in \cite{singh2020improving,singh2022sample}.
It still remains an open problem to derive an analytical expression for the return distribution with continuous state space.  
A challenge that further complicates the problem is represented by the infinitely many decision choices of states.

A related research line is the recent study of RL in the LQR context, which focuses on learning the expected return through interaction with the environment, see \cite{dean2020sample,tu2018least,fazel2018global,malik2019derivative,li2021distributed,yaghmaie2022linear,zheng2021sample}.
For example, \cite{fazel2018global} proposes a model-free policy gradient algorithm for LQR and shows its global convergence with finite polynomial computational and sample complexity.
Moreover, \cite{zheng2021sample} studies model-based RL for the linear quadratic Gaussian (LQG) problem, in which a model is first learnt from data and then used to design a policy. 
In this setup, evaluating the expected return for a policy is easily computed from the Riccati equation, but these methods are not capable of characterising other aspects of distributional information.
Works exploring the distributional information include risk-averse control \cite{van2015distributionally,tsiamis2021linear,kim2021distributional,chapman2021risk,chapman2022optimizing,chapman2021toward,chapman2022risk,kishida2022risk} or distributional robust control \cite{kim2023distributional,hakobyan2022wasserstein,yang2020wasserstein,tacskesen2023distributionally}. However, these methods cannot analyse the return distribution.

\subsection{ Contributions}
This paper aims at studying the return distribution for linear quadratic control problems.
Our contributions are summarised as follows:

\begin{enumerate}
    \item   We provide an analytical expression of the random return for distributional LQR and prove that this return function is a fixed-point solution to the random variable Bellman equation (Theorem~\ref{Thm:exact:dist}). Specifically, we show that the proposed analytical expression consists of infinitely many random variables and holds for arbitrary i.i.d. exogenous disturbances, e.g., non-Gaussian noise or noise with non-zero mean.  
     \item We analyse the variance of the random return and show that it is bounded if the fourth moment of the disturbances is bounded (Theorems~\ref{thm:variance}). Furthermore, we investigate the distributional sensitivity with respect to model perturbations. Under mild assumptions, we show that the maximal difference between the exact and perturbed return distributions can be bounded by the extent of model perturbations (Theorem~\ref{thm:perturb}).
    \item We develop an approximation of the distribution of the random return using a finite number of random variables when the model is known. We show that the maximal difference between the exact and  approximated return distributions decreases linearly with the number of random variables (Theorem~\ref{The:ranretromodelapprox}). In the model-free case, we approximate the return distribution using state trajectories. We show that, with high confidence, the distribution approximation error deceases linearly with respect to the trajectory length and sub-linearly with respect to the number of trajectories (Theorem~\ref{The:ranretromodelfreeapprox}). 
    \item Finally, we derive analytical evidence that most results for distributional LQR have corresponding counterparts for partially observable systems, including exact characterisation of the random return, variance bound, distributional sensitivity under perturbations, and distributional approximation using a finite number of random variables (Corollaries~\ref{The:LQGchara}--\ref{The:LQGapprox}). This provides insight into extending DRL to partially observable systems. 
\end{enumerate}

The work that comes closest to addressing the problems above is our prior work \cite{wang2023policy}:
the current contribution additionally analyses the variance of the random return and the distributional sensitivity with respect to model perturbations. Moreover, this work constructs a confidence bound on the distribution approximation error for the model-free case when the system matrices are unknown. Additionally, we newly derive corresponding counterparts for partially observable models. 

\subsection{Organisation and Notations}

The paper is organized as follows. In Section~\ref{sec:problem}, we provide background on LQR and define our problem. In Section~\ref{sec:LQR}, we provide the main results for distributional LQR, including the analytical expression of the random return, variance bound, distributional sensitivity under perturbations and model-based and model-free distribution approximations. Section~\ref{sec:LQG} provides the main results for partially observable linear systems. In Section~\ref{Sec:simulation}, we experimentally verify our theoretical results.
Finally, we conclude the paper in Section~\ref{sec:conclusion}.

\smallskip

We denote by $\mathbb{R}$ the set of real numbers and $\mathbb{N}$ the set of natural numbers. For a symmetric matrix $P$, the notation $P>0$ means that $P$ is positive definite. 
For a matrix $Q\in \mathbb{R}^{n\times n}$, we denote by $\left\| Q\right\|$ and $ \left\| Q\right\|_F$ the spectral norm and Frobenius norm, respectively. 
To indicate that two random variables $Z_1$ and $Z_2$ are equal in distribution, we use the notation $Z_1 \mathop{=}\limits^{D} Z_2$.
For a random variable $Z$, $\mathbb{E}[Z]$ denotes its expectation.

\vspace{-0.2cm}
\section{Problem Statement}\label{sec:problem}

Consider a discrete-time linear control system:
\begin{align*}
    x_{t+1} = A x_t + B u_t +v_t ,
\end{align*}
where $x_t \in \mathbb{R}^n$, $u_t\in \mathbb{R}^p$, and $v_t \in \mathbb{R}^n$ are the system state, control input, and the exogenous disturbance, respectively. 
We assume that the exogenous disturbances $v_t$ with bounded moments, $t\in \mathbb{N}$, are i.i.d. sampled from a distribution $ \mathcal{D}$ of arbitrary form.

\vspace{-0.2cm}
\subsection{Classical Discounted LQR}
The canonical  LQR problem aims to find a control policy $\pi: \mathbb{R}^n \rightarrow \mathbb{R}^p$ to minimise the objective
\begin{align*}
    J(u) 
    = \mathbb{E}\left[ \sum_{t=0}^{\infty} \gamma^t (x_t^{\rm{T}} Q x_t + u_t^{\rm{T}} R u_t)  \right],
\end{align*}
where $Q,R$ are positive-definite constant matrices and $\gamma \in (0,1)$ is a discount parameter. Given a control policy $\pi$, let $V^{\pi}(x) = \mathbb{E}\left[ \sum_{t=0}^{\infty} \gamma^k (x_t^{\rm{T}} Q x_t + u_t^{\rm{T}} R u_t)  \right]$ denote the expected return from an initial state $x_0 =x$ with $ u_t = \pi( x_t)$.
For the static linear policy $\pi(x_t)=K x_t$, the value function $V^{\pi}(x)$ satisfies the Bellman equation
\begin{align}\label{eq:Bellman expectation}
    V^{\pi}(x) 
    &= x^{\rm{T}} (Q + K^{\rm{T}} R K) x + \gamma \mathop{\mathbb{E}}_{x' = (A+BK)x+v_0 } [V^{\pi}(x')],  
\end{align} 
where the capital letter $x'$ denotes a random variable over which we take the expectation.

When the exogenous disturbance $v_t$ is normally distributed with zero mean, the value function is known to take the quadratic form  $V^{\pi}(x) = x^{\rm{T}} P x +q$, where $P>0$ is the solution of the Lyapunov equation $P = Q+ K^{\rm{T}} R K + \gamma A_K^{\rm{T}} P A_K$ and $q$ is a scalar related to the variance of $v_t$. In particular, the optimal control feedback gain is $K^*=-\gamma(R+\gamma B^{\rm{T}}PB)^{-1}PA$ and $P$ is the solution to the Riccati equation $P = \gamma A^{\rm{T}} P A - \gamma^2 A^{\rm{T}} P B (R+\gamma B^{\rm{T}} P B)^{-1} B^{\rm{T}} P A+Q$.

\subsection{Distributional LQR}
Motivated by the advantages of DRL in better understanding the effects of the randomness in the environment and in considering more general optimality criteria, in this paper we propose a distributional approach to the LQR problem.
Unlike classical RL, which relies on expected returns, DRL \cite{bdr2022} relies on the distribution of random returns, which is referred to \emph{return distribution}.
The return distribution characterises the probability distribution of different returns generated by a given policy and, as such, it contains much richer information on the performance of a given policy compared to the expected return.
In the context of LQR, we denote by $G^{\pi}(x)$ the random return using the static control strategy $u_t = \pi( x_t)$ from the initial state $x_0=x$, which is defined as
\begin{align}
    G^{\pi}(x) 
    = \sum_{t=0}^{\infty} \gamma^t (x_t^{\rm{T}} Q x_t + u_t^{\rm{T}} R u_t) , \nonumber \\
    u_t = \pi( x_t),x_0 = x.
\end{align}
It is straightforward to see that the expectation of $G^{\pi}(x)$ is equal to the value function $V^{\pi}(x)$. 
The standard Bellman equation in \eqref{eq:Bellman expectation} decomposes the long-term expected return into an immediate stage cost plus the expected return of future actions starting at the next step. 
Similarly, we can define the random variable Bellman equation for the random return as 
\begin{align}\label{eq:rv:Bellman}
    G^{\pi}(x) & \mathop{=}^{D} x^{\rm{T}}Qx+\pi(x)^{\rm{T}}R\pi(x) + \gamma G^{\pi}(x'), \nonumber \\
    & x' = Ax+B\pi(x)+v_0.
\end{align}
Here we use the notation $Z_1\mathop{=}\limits^{D} Z_2$ to denote that two random variables $Z_1,Z_2$ are equal in distribution. 
Compared to the expected return in LQR, which is a scalar, here the return distribution is infinite-dimensional. 
%

The following example is used to highlight the need of considering the random return.
\begin{example}
Consider the discrete-time scalar linear system $x_{t+1} = x_t + u_t +v_t$ and three different types of disturbance $v_t$: normal distribution $\mathcal{N}(0,1)$, uniform distribution $U[-\sqrt{3},\sqrt{3}]$, and multimodal distribution which is characterised by the probability density function (PDF) $(p_1(z)+p_2(z))/2$, where  $p_i(z)= \frac{1}{\sqrt{2\pi} \sigma_i} exp\big({-\frac{(z-\mu_i)^2}{2\sigma_i^2}}\big)$, $i=1,2$, with $\mu_1 = -0.99$, $\mu_2 = 0.99$, $\sigma_1=\sigma_2 =\sqrt{1-0.99^2} $. 
Their PDFs are shown in Fig.~\ref{fig:motivation}(a).
It can be verified that the mean of $v_t$ is zero and the variance is $1$ for all three types of disturbance. 
We set the initial state $x=1$, $Q=R=1$, and $\gamma=0.6$.
Then, the optimal controller for the three disturbances is the same and given by $u_t=-0.4684x_t$. The value function $V^{\pi}(x)$  in \eqref{eq:Bellman expectation} of implementing the optimal controller for the three disturbances is the same as well since the variance of $v_t$ is the same.

However, the distribution of the random return $G^{\pi}(x)$ varies significantly for the three disturbances, as shown in Fig.~\ref{fig:motivation}(b). 
We observe that the distributions of the random cost for the Gaussian and uniform disturbances are close to chi-square distributions, but the distribution for the multimodal disturbance exhibits multiple peaks.
Hence, the distribution of the random return contains more information than the LQR problem. From $G^{\pi}(x)$, it is possible, for instance, to evaluate the risk of incurring high costs, which is not possible to do from the value function $V^{\pi}(x)$.
\end{example}

\begin{figure}[t]
    \centering
    \begin{subfigure}[b]{0.24\textwidth}
        \centering
        \includegraphics[width=\textwidth]{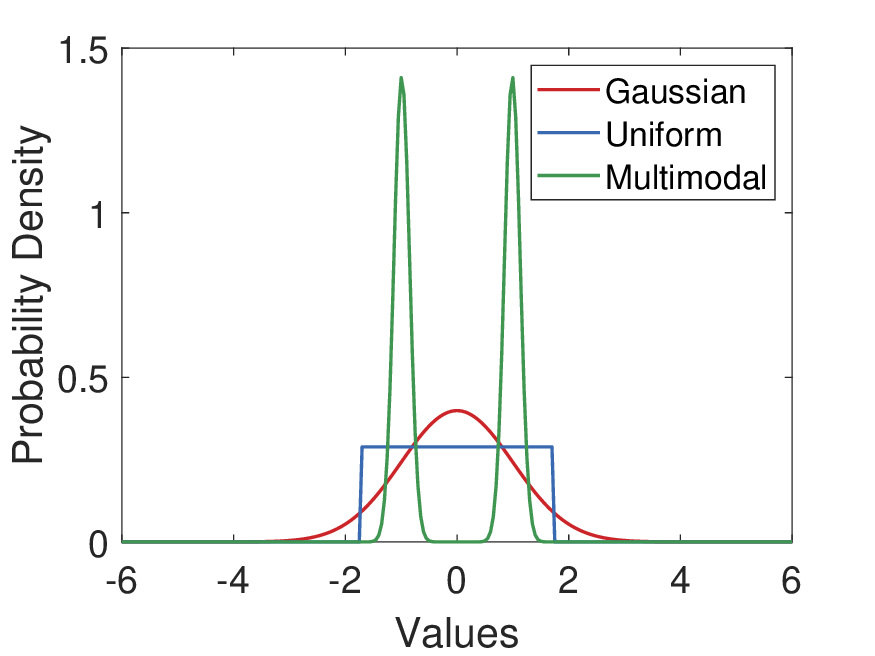}
        \caption{Probability density values of three types of disturbance $v_t$. \\}
        \label{fig:m:f1}
    \end{subfigure}
    \hfill
    \begin{subfigure}[b]{0.24\textwidth}
        \centering
        \includegraphics[width=\textwidth]{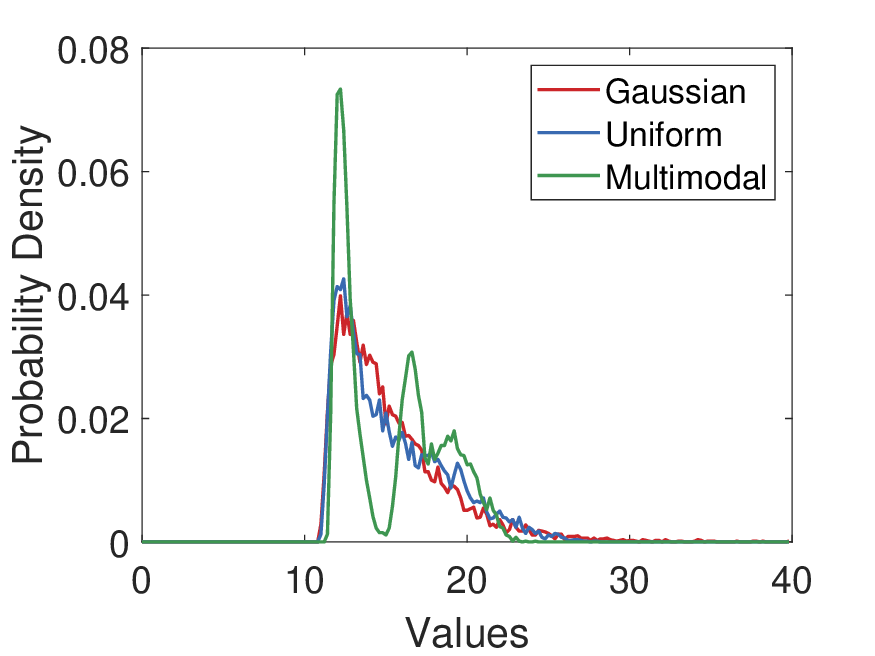}
        \caption{Probability density values of the random cost $G^{\pi}(x)$ induced by three disturbances.}
        \label{fig:m:f2}
    \end{subfigure}
    \caption{The PDFs of three types of disturbance and of their corresponding random costs in LQR. The PDFs of the random costs are generated by Algorithm~\ref{alg:algorithm_MFPE} in this paper. } \label{fig:motivation}
\end{figure}

This paper addresses the following research problems. We first analytically characterise the random return that fulfils the random variable Bellman equation for LQR. Subsequently, we explore the fundamental properties of the random return and its distribution: variance bound, distributional sensitivity under perturbations, and model-based and model-free distribution approximations. Finally, we extend our investigation to encompass partially observable linear systems.

\section{Main Results on Distributional LQR}\label{sec:LQR}
This section focuses on the random return for the LQR problem.
\subsection{Characterisation of the Random Return}
In this section, we precisely characterise the distribution of the random return  that satisfies the distributional  Bellman equation \eqref{eq:rv:Bellman}. 
Given a static linear policy $\pi(x_t)=K x_t$, we denote by $G^K(x)$ the random return $G^{\pi}(x)$  under the policy $\pi(x_t)$ from the initial state $x_0=x$ , which is defined as
\begin{align}\label{eq:Gk:def:original}
    G^{K}(x) =  \sum_{t=0}^{\infty} \gamma^t x_t^{\rm{T}} (Q+K^{\rm{T}} R K) x_t  , \quad x_0 = x.
\end{align}
The random return  $G^{K}(x)$ satisfies the following random variable Bellman equation  
\begin{align}\label{eq:rv:bellman}
    G^{K}(x) & \mathop{=}^{D} x^{\rm{T}}Q_Kx + \gamma G^{K}(x'),\quad x'=A_K x+ v_0,
\end{align}
where $A_K:=A+BK$ and $Q_K := Q+ K^{\rm{T}} R K$. 
In the following theorem, we provide an explicit expression of the random return $G^K(x)$. The proof can be found in Appendix~\ref{app:thm:exact:dist}.

\begin{theorem}\label{Thm:exact:dist}
Suppose that the feedback gain  $K$ is stabilizing, i.e., $A_K=A+BK$ is stable. Let  \begin{align}\label{eq:dist_func}
    &G^{K}(x) = x^{\rm{T}} P x  +  2 \sum_{k = 0 }^{\infty} \gamma^{k+1} w_k^{\rm{T}} P A_K^{k+1}x \nonumber \\
    &+  \sum_{k = 0 }^{\infty}  \gamma^{k+1}  w_k^{\rm{T}} P w_k+ 2  \sum_{k = 1 }^{\infty} \gamma^{k+1} w_k^{\rm{T}} P \sum_{\tau=0}^{k-1} A_K^{k-\tau}w_{\tau},
\end{align}
where  $P$ is obtained from the Lyapunov equation $P = Q+ K^{\rm{T}} R K + \gamma A_K^{\rm{T}} P A_K$, and the random variables $w_k \sim \mathcal{D} $ are independent from each other for all $k\in\mathbb{N}$. Then, the random variable $G^{K}(x)$ defined in 
    \eqref{eq:dist_func} is a fixed point solution to the random variable Bellman equation \eqref{eq:rv:bellman}.
\end{theorem}

We note that the expression of the random return is meaningful only when the system is stable. When ensuring stability, this analytical expression applies to arbitrary exogenous disturbances including non-Gaussian, uniform noises and noises with non-zero means, as long as the disturbances are i.i.d..

If we assume $\mathbb{E}[w_k]=0$, $\mathbb{E}[w_k w_k^{\rm{T}}] = \sigma^2 I$, and the disturbances $w_k$ are i.i.d., we have that the expected value of the random return is $x^{\rm{T}} P x + \sigma^2 \frac{\gamma}{1-\gamma} {\rm{Tr}}(P)$, which aligns with the classical result in LQR. This observation to some degree validates our characterisation of the random return.

\begin{remark}
Recall that the PDF of the sum of two independent random variables is the convolution of their two PDFs.
Computing the accurate probability distribution function of $G^{K}(x)$ in \eqref{eq:dist_func} is a challenging task due to the potential need for an infinite number of convolution operations.
However, we can discuss the approximate shape of this distribution under different conditions.
Suppose that the random variable $w_k$ follows a normal distribution.
When the initial state is significantly large, the random variable $\sum_{k = 0 }^{\infty} \gamma^{k+1} w_k^{\rm{T}} P A_K^{k+1}x$ dominates the random return.
This sum follows a Gaussian distribution, and as a result, the overall distribution of $G^{K}(x)$ tends to resemble a Gaussian distribution.
Conversely, when the value of $x$ is small, the term $\sum_{k = 0 }^{\infty} \gamma^{k+1} w_k^{\rm{T}} P w_k$ becomes dominant. This sum follows a chi-square distribution, and consequently, the entire distribution of $G^{K}(x)$ takes on a chi-square-like shape. More details can be found in Section~\ref{Sec:simulation}.
\end{remark}

\subsection{Bounded Variance of the Random Return} 
In this section, we analyse the variance of the random return $G^K(x)$, which is presented in the following theorem. The proof can be found in Appendix~\ref{app:thm:variance}.
\begin{theorem}\label{thm:variance}
Assume that $\mathbb{E}[w_k]=0$ and $\mathbb{E}[\left\|w_k \right\|^4] \leq \sigma_4^4$, for all $k\in\mathbb{N}$.
Suppose that the feedback gain $K$ is stabilizing such that $\left\|A_K\right\| = \rho_K <1$.
Then, the variance of the random variable $G^K(x)$ is bounded.
\end{theorem}
Although $G^K(x)$ in \eqref{eq:dist_func} is composed of infinitely many random variables, Theorem~\ref{thm:variance} shows that its variance is bounded if the fourth moment of the disturbance is bounded. 
The fourth moment qualitatively is a measure of the tail of a probability distribution. To ensure a finite variance for the random cost $G^K(x)$, 
we thus require that the tail of the disturbance distribution is not heavy. This condition seems indispensable, since $G^K(x)$ includes a term $w_k^{\rm{T}} P w_k$.

\subsection{Sensitivity Analysis of the Return Distribution} 

In this section, we investigate how perturbations on matrices influence the distribution of the random return $G^{K}(x)$. 
Suppose that we perturb the matrices $A,B$ by an amount $\Delta A$, $\Delta B$, respectively. 
Let
\begin{align*}
    &\tilde{A}=A+\Delta A, \quad \tilde{B}=B + \Delta B, \quad A_K = A+BK,\nonumber \\
    &  \tilde{A}_K = \tilde{A} + \tilde{B} K, \quad \Delta A_K = \tilde{A}_K - A_K,
\end{align*}
and let $P$ and $\tilde{P}$ be the solutions to 
\begin{align*}
    & P - Q-K^{\rm{T}}R K = \gamma A_K^{\rm{T}} P A_K, \\
    &\tilde{P}- Q-K^{\rm{T}}R K = \gamma \tilde{A}_K^{\rm{T}} \tilde{P} \tilde{A}_K,
\end{align*}
respectively.
With the introduction of perturbations on matrices, we define the perturbed random variable 
\begin{align}\label{eq:dist_func_perturbed}
    &\tilde{G}^{K}(x) = x^{\rm{T}} \tilde{P} x +  2 \sum_{k = 0 }^{\infty} \gamma^{k+1} w_k^{\rm{T}} \tilde{P} \tilde{A}_K^{k+1}x\nonumber \\
    & +  \sum_{k = 0}^{\infty}  \gamma^{k+1}  w_k^{\rm{T}} \tilde{P} w_k 
    + 2  \sum_{k = 1 }^{\infty} \gamma^{k+1} w_k^{\rm{T}} \tilde{P} \sum_{\tau=0}^{k-1} \tilde{A}_K^{k-\tau}w_{\tau}.
\end{align}
Let $F^{K}_x$ and $\tilde{F}^{K}_{x}$ denote the cumulative distribution function (CDF) of $G^{K}(x)$ and $\tilde{G}^{K}(x)$, respectively.
In the following theorem, we show that the sup difference between $F^{K}_x$ and $\tilde{F}^{K}_{x}$ is bounded when the perturbation is reasonably small.
The proof can be found in Appendix~\ref{app:thm:perturb}.
\begin{theorem}\label{thm:perturb}
Assume that the PDF of $w_k$ is bounded, and satisfies $\mathbb{E}[w_k^{\rm{T}} w_k] \leq \sigma^2$, for all $k\in \mathbb{N}$.
Suppose that the feedback gain $K$ is stabilizing such that $ \max \{ \big\|A_K\big\|,\big\| \tilde{A}_K\big\|\} =  \rho_K <1$.
Suppose that $l>2 \epsilon$, where $l = \left\| H^{-1}\right\|^{-1}$, $H = I \otimes I - \gamma A_K^{\rm{T}} \otimes A_K^{\rm{T}}$ and $\epsilon = \gamma \left\| A_K \right\|_F \left\| \Delta A_K \right\|_F + \frac{\gamma}{2} \left\| \Delta A_K \right\|_F^2$. Then, we have
\begin{align}\label{eq:pert:bound}
     &\sup_{z}|{F}^{K}_x(z)-\tilde{F}^{K}_{x}(z) | \leq \tilde{c}_1 \left\|\Delta A_K \right\| + \tilde{c}_2 \left\| \Delta A_K \right\|^2,
\end{align}
where the constants $\tilde{c}_1, \tilde{c}_2$ (made explicit in the proof) depend on the system matrices, the initial state value $x$, and the parameters $\gamma,\rho_K,\sigma$.
\end{theorem}

Traditional sensitivity analysis investigates the impact of perturbations on solutions to the Lyapunov equation, see, e.g., \cite{laub1990sensitivity}. Building on this result, Theorem~\ref{thm:perturb} shows that we can also bound changes in the perturbed return distribution.

%

\subsection{Model-Based Approximation of the Return Distribution}\label{Sec:Approxreturn}
The expression of the random return $G^{K}(x)$ defined in \eqref{eq:dist_func} is composed of infinitely many random variables.
In this section, we investigate how to approximate the distribution of this random return  using a finite number of random variables. A natural idea is to consider only the first $N$ terms in the summations in the expression  \eqref{eq:dist_func}  and disregard the terms for $k$ larger than $N$, which yields the following:   
\begin{align}\label{eq:approxreturn}
    &{G}^{K}_N(x) =  x^{\rm{T}} P x  +  2 \sum_{k = 0 }^{N-1} \gamma^{k+1} w_k^{\rm{T}} P A_K^{k+1}x \nonumber \\
    &+  \sum_{k = 0 }^{N-1}  \gamma^{k+1}  w_k^{\rm{T}} P w_k + 
 2 \sum_{k = 1 }^{N-1} \gamma^{k+1} w_k^{\rm{T}} P \sum_{\tau=0}^{k-1} A_K^{k-\tau}w_{\tau}.
\end{align}
Let ${F}^{K}_{x,N}$ denote the CDF of ${G}_N^{K}(x)$. The following theorem provides an upper bound on the difference between $F^{K}_x$ and ${F}^{K}_{x,N}$, and shows that  the sequence $\{{G}^{K}_N(x)\}_{N\in \mathbb{N}}$ converges pointwise in distribution to $G^{K}(x)$,  $\forall x\in\mathbb{R}^n$. 
The proof can be found in Appendix~\ref{app:thm:modelappro}.

\begin{theorem}\label{The:ranretromodelapprox}
Assume that the PDF of $w_k$ is bounded, and satisfies $\mathbb{E}[w_k^{\rm{T}} w_k] \leq \sigma^2$, for all $k\in \mathbb{N}$.
Suppose that the feedback gain  $K$ is stabilizing such that $\left\|A_K\right\| = \rho_K <1$. Then, the sup difference between the CDFs $F^K_x$ and ${F}^{K}_{x,N}$ is bounded by
\begin{align}\label{eq:approximat:bound}
     \sup_{z}|{F}^{K}_x(z)-{F}^{K}_{x,N}(z) | \leq c_0 \gamma^N,
\end{align}
where $c_0$ is a constant (again, made explicit in the proof) that depends on the system matrices, the initial state value $x$, and the parameters $\gamma, \rho_K, \sigma$. 
\end{theorem}


\begin{remark}
The bound on the distribution approximation in \eqref{eq:approximat:bound} relies on the conditions of Theorem~\ref{The:ranretromodelapprox}, which ensure that the PDF of $G^K_N$ is continuous and bounded. Note that these conditions are not strict, and indeed hold for many  noise distributions commonly used in linear control systems, including the Gaussian and uniform ones.  
\end{remark}

\subsection{Model-Free Approximation of the Return Distribution}
When the matrices $A,B$ are unknown, one cannot use the exact form of the random return to compute the distribution. In this section, we propose a model-free method to estimate the distribution of the random return.

In the absence of information about the system matrices A and B, one costly yet straightforward approach to estimate the distribution is by directly sampling the random return $G^{K}(x)$ as defined in \eqref{eq:Gk:def:original}.
This random return represents the sum of discounted rewards over an infinite time horizon.
To make the computation practically manageable, we truncate the time horizon and disregard rewards occurring after time step $T$.
Accordingly, we define the random variable
\begin{align}
    G^{K,T}(x) =  \sum_{t=0}^{T} \gamma^t x_t^{\rm{T}} (Q+K^{\rm{T}} R K) x_t  , \quad x_0 = x.
\end{align}
We denote by $F^{K,T}_x(z)$ the CDF of $G^{K,T}(x)$ and recall that $F^{K}_x(z)$ is the CDF of $G^K(x)$.
Intuitively, $F^{K,T}_x(z)$ closely approximates $F^{K}_x(z)$ when $T$ is sufficiently large.
This is due to the fact that every term beyond time step $T$ becomes negligible after being discounted by $\gamma^t$.
Therefore, we sample the random return $G^{K,T}(x)$ to estimate the distribution of $G^{K}(x)$.

\begin{algorithm}[t]
\caption{Model-free Distributional Policy Evaluation}
\begin{algorithmic}[1]
\REQUIRE  initial values $x$, controller $K$
    \FOR {${\rm{iteration}} \; m=1,\ldots,M$}
        \STATE Initial state $x_{m,0}=x$;
        \FOR {$ {\rm{time}} \; t=0,1,\ldots,T-1$}
            \STATE Implement controller $u_{m,t} = K x_{m,t}$;
            \STATE Observe $x_{m,t+1}=A x_{m,t} +B u_{m,t} + v_t$;
        \ENDFOR
        \STATE Obtain ${G}^{K,T}_m(x) =  \mathop\sum\limits_{t=0}^{T} \gamma^t x_{m,t}^{\rm{T}} (Q+K^{\rm{T}} R K) x_{m,t}$;
    \ENDFOR
    \STATE Construct EDF $\hat{F}^{K,T}_{x,M}(z)= \frac{1}{M} \mathop\sum\limits_{m=1}^M \textbf{1}\{ {G}^{K,T}_m(x) \leq z\}$.
\end{algorithmic}\label{alg:algorithm_MFPE}
\end{algorithm}

The detailed model-free distributional policy evaluation is presented in Algorithm~\ref{alg:algorithm_MFPE}.
Specifically, at each iteration $m$, starting from the initial state $x_{m,0} = x$, we repeatedly implement the static controller $u_{m,t} = K x_{m,t}$ and generate a total of $M$ trajectories.
Given the $m$-th trajectory $\{ x_{m,t}\}_{t=0:T}$ with $x_{m,0}=x$, define the sampling cost
\begin{align}
    {G}^{K,T}_m(x) =  \sum_{t=0}^{T} \gamma^t x_{m,t}^{\rm{T}} (Q+K^{\rm{T}} R K) x_{m,t}.
\end{align}
Using ${G}^{K,T}_m(x)$, $m=1,\ldots,M$, the empirical distribution function (EDF) is constructed by
\begin{align}
    \hat{F}^{K,T}_{x,M}(z) = \frac{1}{M} \sum_{m=1}^M \textbf{1}\{ {G}^{K,T}_m(x) \leq z\},
\end{align}
where $\textbf{1}\{ \cdot\}$ denotes the indicator function.
Intuitively, the empirical distribution $\hat{F}^{K,T}_{x,M}(z)$ is close to the distribution $F^{K,T}_x(z)$ when $M$ is sufficiently large and to the distribution $F^K_x(z)$ when $T$ is also large.
The following theorem provides an upper bound on the difference between the distributions $\hat{F}^{K,T}_{x,M}(z)$ and $F^K_x(z)$.
The proof can be found in Appendix~\ref{app:thm:modelfreeappro}.

\begin{theorem}\label{The:ranretromodelfreeapprox}
Assume that the PDF of $w_k$ is bounded, and satisfies  $\mathbb{E}[w_k]=0$ and $\mathbb{E}[w_k^{\rm{T}} w_k] \leq \sigma^2$, for all $k\in \mathbb{N}$. Suppose that the feedback gain $K$ is stabilizing such that $\left\|A_K\right\| = \rho_K <1$. Then, with probability at least $1-\delta$, we have 
\begin{align}\label{eq:model_free_appro_error}
    \sup_z| \hat{F}_{x,M}^{K,T}(z) - F^K_x(z)| \leq \sqrt{\frac{\ln(1 /\delta)}{2M}} \nonumber \\+ f_{\max}\left\|Q_K \right\| \gamma^{T+1}(c_1 \rho_K^{2(T+1)}  + c_2 \rho_K^{T+1}  + c_3),
\end{align}
where 
$f_{\max}$ is the maximum of the PDF of the random variable $G^{K,T}(x)$, $ Q_K =  Q+ K^{\rm{T}}RK $,
$c_1= \frac{\left\| x\right\|^2}{1-\gamma \rho_K^2}$, $c_2 = \frac{2 \left\| x\right\| \sigma}{(1-\rho_K)(1- \gamma \rho_K)}$, $c_3 = \frac{\sigma^2}{(1-\gamma)(1-\gamma \rho_K) }$.
\end{theorem}
Theorem~\ref{The:ranretromodelfreeapprox} shows that the accuracy of the distribution estimate depends on the choice of two key parameters: the time horizon $T$ and the number of generated trajectories $M$.
When both $M$ and $T$ are sufficiently large, $\hat{F}_{x,M}^{K,T}(z)$ can serve as a reliable approximation for $F^K_x(z)$.
Supported by this result, we consider the return distribution learned by Algorithm~\ref{alg:algorithm_MFPE} with sufficiently large values of $M,T$ as the true return distribution in the simulation part.

\begin{table}
    \centering
    \begin{tabular}{ccccc}
        \toprule
         $\gamma$ & UB &  $N$  & $T,M(1-\delta=95\%)$ & $T,M(1-\delta=99\%)$  \\
        \midrule
         0.6& 0.02 & 11 &   (100,4000) &(100,6000)    \\
         0.6& 0.01 & 12 &   (100,15000) &(100,23000)    \\
         0.8& 0.02 & 24 &   (100,4000) &(100,6000)    \\
         0.8& 0.01 & 27 &   (100,15000) &(100,23000)    \\
         \bottomrule
    \end{tabular}
    \caption{ Comparison of model-based (left) and model-free (right) approximation methods. Here, $\gamma$ is the discount parameter; UB is the bound in the right hand side of \eqref{eq:approximat:bound} and \eqref{eq:model_free_appro_error} for model-based and model-free methods, respectively; $N$ is the smallest integer such that $c_0 \gamma^N \leq {\rm{UB}}$; and $T,M(1-\delta)$ denote the pair comprising the horizon length $T$ and the number of trajectories $M$ required so that the approximation error in \eqref{eq:model_free_appro_error} is bounded by UB with probability at least $1-\delta$. }
    \label{tab:modelfree}
\end{table}

\begin{remark}
We note that the random variables $G^K_N(x)$ and $G^{K,T}(x)$ serve as the approximations to the true random return $G^K(x)$ by truncating the number of random variables and the time horizon, respectively. As shown in Theorems~\ref{The:ranretromodelapprox} and~\ref{The:ranretromodelfreeapprox}, increasing the number of random variables or extending the truncated horizon definitely enhance the approximation accuracy for model-based and model-free methods, respectively. However, the associated costs one needs to pay to obtain their distributions are usually different. 
As shown in Table~\ref{tab:modelfree}, the model-free method requires a sufficiently large value of $M$ to achieve a reliable distribution estimate with a high probability. In contrast,
the model-based method can attain the same level of accuracy with probability $1$ using a significantly smaller number of random variables.
Hence, when the system matrices and the disturbances are known, the computation of the distribution of $G^K_N(x)$ incurs less costs. It is also worth noting that the model-free method is not sensitive to the discount parameter $\gamma$ while the model-based method is.
\end{remark}

\section{Extension to Partially Observable Systems}\label{sec:LQG}
In this section, we analyse the case when the state is not fully observable. We show that most of the results for LQR can be extended to this partially observable case.

Consider a partially observable discrete-time linear control system:
\begin{align}
    x_{t+1}&=A x_t + B u_t + v_t,  \nonumber \\
    y_t &= C x_t + s_t, \nonumber 
\end{align}
where $x_t \in \mathbb{R}^n$, $u_t\in \mathbb{R}^p$, $y_t\in \mathbb{R}^l$,  $v_t \in \mathbb{R}^n$, and $s_t\in \mathbb{R}^l$ are the system state, control input, system output,  process noise,  and observation noise, respectively. We assume that the system is observable and controllable.
By introducing the feedback gain $K$ and the observer gain $L$, we define the estimated state  and controller 
\begin{align}
    \hat{x}_{t+1} &= A \hat{x}_t + B u_t + L(y_t - C  \hat{x}_t), \nonumber \\
     u_t& = K \hat{x}_t. \nonumber
\end{align}
By defining $\tilde{x}_t = x_t - \hat{x}_t$, $\bar{x}_t = [x_t^{\rm{T}}, \tilde{x}_t^{\rm{T}}]^{\rm{T}}$, we get the augmented system
\begin{align}\label{eq:dynamics:augmented}
    \bar{x}_{t+1} = \bar{A}_{KL} \bar{x}_t +  \bar{v}_t,
\end{align}
where
\begin{align*}
    \bar{A}_{KL} &= \left[ \begin{array}{cc}
         A+BK& -BK  \\
         0& A-LC 
    \end{array} \right],
    \bar{v}_t = F \left[ \begin{array}{c} v_t \\ s_t \end{array} \right],\\
    F &= \left[ \begin{array}{cc}
         I& 0  \\
         I& -L 
    \end{array} \right].
\end{align*}
If $\mathbb{E}[v_t]=\mathbb{E}[s_t]=0$, and the collection of $v_t$ and $s_t$ is i.i.d., it is easy to verify that $\mathbb{E}[\bar{v}_t]=0$ and the collection of $\bar{v}_t$ is i.i.d.. 
We denote the distribution of $\bar{v}_t$ by $\bar{\mathcal{D}}$. 
Define $\bar{Q}_K:=\left[ \begin{array}{cc} Q + K^{\rm{T}} R K & -K^{\rm{T}} R K  \\ -K^{\rm{T}} R K & K^{\rm{T}} R K  \end{array} \right]$ and the random return 
\begin{align}\label{eq:dist_func_partial}
    G^{KL}(\bar{x}) &=  \sum_{t=0}^{\infty} \gamma^t (x_t^{\rm{T}} Q x_t + u_t^{\rm{T}} R u_t ) \nonumber \\
    &= \sum_{t=0}^{\infty} \gamma^t
    (x_t^{\rm{T}} Q x_t + \hat{x}_t^{\rm{T}} K^{\rm{T}} R K \hat{x}_t ) \nonumber \\
    &= \sum_{t=0}^{\infty} \gamma^t \bar{x}_t^{\rm{T}} \bar{Q}_K \bar{x}_t , \quad \bar{x}_0 = \bar{x},
\end{align}
where $\bar{x}=[x^{\rm{T}},\tilde{x}_0^{\rm{T}}]^{\rm{T}}$.
The random return  $G^{KL}(\bar{x})$ satisfies the following random variable Bellman equation  
\begin{align}\label{eq:rv:bellman:LQG}
    G^{KL}(\bar{x}) & \mathop{=}^{D} \bar{x}^{\rm{T}} \bar{Q}_K \bar{x} + \gamma G^{KL}({X}'),\quad {X}'=\bar{A}_{KL} \bar{x}+ \bar{v}_0.
\end{align}


In the following corollary, we provide an explicit expression of the random return $G^{KL}(\bar{x})$. The proof can be obtained by applying the results in Theorem~\ref{Thm:exact:dist} to the augmented system \eqref{eq:dynamics:augmented} and is omitted.
\begin{corollary}\label{The:LQGchara}
Suppose that the feedback gain $K$ and observer gain $L$ are chosen such that $\bar{A}_{KL}$ is stable. Let  
\begin{align}\label{eq:dist_func:LQG}
    &G^{KL}(\bar{x}) = \bar{x}^{\rm{T}} \bar{P} \bar{x} +  2 \sum_{k = 0 }^{\infty} \gamma^{k+1} \bar{w}_k^{\rm{T}} \bar{P} \bar{A}_{KL}^{k+1}\bar{x} \nonumber \\
    &+  \sum_{k = 0 }^{\infty}  \gamma^{k+1}  \bar{w}_k^{\rm{T}} \bar{P} \bar{w}_k + 
 2  \sum_{k = 1 }^{\infty} \gamma^{k+1} \bar{w}_k^{\rm{T}} \bar{P} \sum_{\tau=0}^{k-1} \bar{A}_{KL}^{k-\tau}\bar{w}_{\tau},
\end{align}
where $\bar{P}$ is obtained from the Lyapunov equation $\bar{P} = \bar{Q}_K + \gamma \bar{A}_{KL}^{\rm{T}} \bar{P} \bar{A}_{KL}$, and the random variables $\bar{w}_k \sim \bar{\mathcal{D}} $ are independent from each other for all $k\in\mathbb{N}$. Then, the random variable $G^{KL}(x)$ defined in \eqref{eq:dist_func_partial} is a fixed point solution to the random variable Bellman equation \eqref{eq:rv:bellman:LQG}.
\end{corollary}

The variance bound part is similar to that of the fully observable case, and is presented in the following corollary. The proof can be obtained by following a similar methodology to that employed for Theorem~\ref{thm:variance} and is omitted.
\begin{corollary}\label{The:LQGvariance}
Assume that $\mathbb{E}[\bar{w}_k] = 0$ and $\mathbb{E}[\left\| \bar{w}_k \right\|^4] \leq \bar{\sigma}_4$, for all $k\in \mathbb{N}$. Suppose that the feedback gain  $K$ and observer gain $L$ are chosen such that $\left\|\bar{A}_{KL}\right\| = \bar{\rho}_K <1$. Then, the variance of the random variable $G^{KL}(x)$ is bounded.
\end{corollary}

The sensitivity analysis for LQG is similar to that of the fully observable case.
Suppose that we perturb the matrix $\bar{A}_{KL}$ by an amount $\Delta \bar{A}_{KL}$. Define the matrix $\check{{A}}_{KL} = \bar{A}_{KL} + \Delta \bar{A}_{KL}$ and the perturbed random variable 
\begin{align}\label{eq:dist_func_perturbed:LQG}
    &\tilde{G}^{KL}(\bar{x}) = \bar{x}^{\rm{T}} \check{{P}} \bar{x} +  2 \sum_{k = 0 }^{\infty} \gamma^{k+1} \bar{w}_k^{\rm{T}} \check{{P}} \check{{A}}_K^{k+1}x\nonumber \\
    & +  \sum_{k = 0}^{\infty}  \gamma^{k+1}  \bar{w}_k^{\rm{T}} \check{{P}} \bar{w}_k 
    + 2  \sum_{k = 1 }^{\infty} \gamma^{k+1} \bar{w}_k^{\rm{T}} \check{{P}} \sum_{\tau=0}^{k-1} \check{{A}}_K^{k-\tau} \bar{w}_{\tau}.
\end{align}
where $\check{{P}} = \bar{Q}_K + \gamma \check{{A}}_{KL}^{\rm{T}} \check{{P}} \check{{A}}_{KL}$.
Let $F^{KL}_x$ and $\tilde{F}^{KL}_{x}$ denote the CDF of $G^{KL}(x)$ and $\tilde{G}^{KL}(x)$, respectively.
We obtain the perturbation for partially observable case in the following corollary. The proof can be adapted from that of Theorem~\ref{thm:perturb} and is omitted.
\begin{corollary}\label{The:LQGrobustness}
Assume that the PDF of $\bar{w}_k$ is bounded, and satisfy $\mathbb{E}[\bar{w}_k^{\rm{T}} \bar{w}_k] \leq \bar{\sigma}^2$, for all $k\in \mathbb{N}$.
Suppose that the feedback gain $K$ and the observer $L$ are chosen such that $ \max \{ \left\|\bar{A}_{KL}\right\|,\left\| \check{{A}}_{KL} \right\|\} =  \bar{\rho}_K <1$.
Suppose that $\bar{l}>2 \bar{\epsilon}$, where $\bar{l} = \left\| \bar{H}^{-1}\right\|^{-1}$, $\bar{H} = I \otimes I - \gamma \bar{A}_{KL}^{\rm{T}} \otimes \bar{A}_{KL}^{\rm{T}}$ and $\bar{\epsilon} = \gamma \left\| \bar{A}_{KL} \right\|_F \left\| \Delta \bar{A}_{KL} \right\|_F + \frac{\gamma}{2} \left\| \Delta \bar{A}_{KL} \right\|_F^2$. Then, we have
\begin{align*}
     &\sup_{z}|{F}^{KL}_x(z)-\tilde{F}^{KL}_{x}(z) | \leq \bar{c}_1 \left\|\Delta \bar{A}_{KL} \right\| + \bar{c}_2 \left\| \Delta \bar{A}_{KL} \right\|^2,
\end{align*}
where the constants $\bar{c}_1, \bar{c}_2$ depend on the system matrices, the initial state value $\bar{x}$, and the parameters $\gamma,\bar{\rho}_K,\bar{\sigma}$.
\end{corollary}

The approximation part is similar to that of the fully observable case.
Let
\begin{align}\label{eq:dist_func:appro:LQG}
    &G^{KL}_N(\bar{x}) = \bar{x}^{\rm{T}} \bar{P} \bar{x} +  2 \sum_{k = 0 }^{N} \gamma^{k+1} \bar{w}_k^{\rm{T}} \bar{P} \bar{A}_{KL}^{k+1}\bar{x} \nonumber \\
    &+  \sum_{k = 0 }^{N}  \gamma^{k+1}  \bar{w}_k^{\rm{T}} \bar{P} \bar{w}_k + 
 2  \sum_{k = 1 }^{N} \gamma^{k+1} \bar{w}_k^{\rm{T}} \bar{P} \sum_{\tau=0}^{k-1} \bar{A}_{KL}^{k-\tau}\bar{w}_{\tau}.
\end{align}
Let $F^{KL}_x$ and ${F}^{KL}_{x,N}$ denote the CDF of $G^{KL}(x)$ and $G^{KL}_N(x)$, respectively.
In the following theorem, we show that the approximation error with a finite number of random variables can be bounded in the partially observable case. The proof can be adapted from that of Theorem~\ref{The:ranretromodelapprox} and is omitted.
\begin{corollary}\label{The:LQGapprox}
Assume that the PDF of $\bar{w}_k$ is bounded, and satisfy $\mathbb{E}[\bar{w}_k^{\rm{T}} \bar{w}_k] \leq \bar{\sigma}^2$, for all $k\in \mathbb{N}$.
Suppose that the feedback gain  $K$ and observer gain $L$ are chosen such that $\left\| \bar{A}_{KL}\right\| = \bar{\rho}_K <1$.
Then, the sup difference between the CDFs $F^{KL}_x$ and ${F}^{KL}_{x,N}$ is bounded by
\begin{align}\label{eq:approximat:bound:LQG}
     \sup_{z}|{F}^{KL}_x(z)-{F}^{KL}_{x,N}(z) | \leq \bar{c}_0 \gamma^N,
\end{align}
where $\bar{c}_0$ is a constant that depends on the system matrices, the initial state value $\bar{x}$, and the parameters $\gamma, \bar{\rho}_K, \bar{\sigma}$.
\end{corollary}

We remark that the model-free approximation (Theorem~\ref{The:ranretromodelfreeapprox}) is not applicable for partially observable systems. Unlike the distributional LQR where the state is directly measurable, it is nontrivial to achieve an accurate estimation of the states such that the cumulative estimation error can be controlled arbitrarily small using only the observation sequence $\{y_t\}$ and control sequence $\{u_t\}$ (when the system model is unknown). Thus, we put this problem as a future direction.

\section{Experiments}\label{Sec:simulation}
In this section, we consider an idealized example of data center cooling with three sources coupled to their own cooling devices \cite{yaghmaie2022linear,recht2019tour,dean2020sample} with the dynamics 
$ x_{t+1} = A x_t + B u_t +v_t$, where  
\begin{align*}
    &A = \left[ \begin{array}{ccc}
         1.01&0.01&0  \\
         0.01&1.01&0.01 \\
         0&0.01& 1.01 
    \end{array} \right],
    B =  \left[ \begin{array}{ccc} 1 &0&0 \\ 0&1&0 \\ 0&0&1 \end{array} \right].
\end{align*}
We select $Q=I$ and $R=I$. The exogenous disturbances have standard normal distributions with zero mean. 

Even for this linear system, it is impossible to simplify the expression of the exact return distribution, which still depends on an infinite number of random variables. 
Thus, as a baseline for the return distribution, we generate an empirical distribution by Algorithm~\ref{alg:algorithm_MFPE} with a sufficiently large amount of samples that approximates the true distribution of the random return.
Specifically, we run Algorithm 1 with the parameters $T=3000$ and $M=30000$. By Theorem~\ref{The:ranretromodelfreeapprox}, the maximal difference between the generated empirical distribution and the true one is bounded by $0.0088$ with probability at least $99\%$, which means that the generated empirical distribution is reliably close to the true one.
We use the sample frequency over evenly-divided regions as an approximation of the PDF.

\subsection{LQR}
We first consider the fully observable case. 
We select different values of $\gamma$ and $x_0$, and fix the optimal controller gain  $K=-\gamma(R+\gamma B^TPB)^{-1}PA$, where $P$ is the solution to the classic Riccati equation $P = \gamma A^T P A - \gamma^2 A^T P B (R+\gamma B^T P B)^{-1} B^T P A+Q$. The controller is given by
$$K=-0.01\left[\begin{array}{ccc} 56.19&0.7692&0.0027\\0.7692&56.20&0.7692\\0.0027&0.7692&56.19\end{array}\right].$$

In what follows, we verify the results in Theorem~\ref{The:ranretromodelapprox} by evaluating the quality of the approximation of the return distribution using different numbers of random variables. 
We denote here by $f_N$ the distribution of the approximated random return $G^K_N(x_0)$ in \eqref{eq:approxreturn} obtained considering $N$ random variables.
We compute the constant $c_0$ in equation \eqref{eq:approximat:bound} and the required number of random variables that guarantees $\sup_{z}|{F}^{K}_x(z)-{F}^{K}_{x,N}(z) | \leq 0.01$, meaning that the estimate distribution is sufficiently close to the true distribution. As shown in Table~\ref{tab:my_label}, an increasing number of random variables is needed when dealing with larger values of $\gamma$ and/or $x_0$.
The simulation results are shown in Fig.~\ref{fig:approximation}.
Specifically, Fig.~\ref{fig:approximation} (a) and (c) show that when $\gamma$ is small, the return distribution can be well approximated using only a few random variables ($N=7$ works well). 
However, when $\gamma$ approaches $1$, more random variables are needed for an accurate approximation: as shown in Fig.~\ref{fig:approximation} (b) and (d), we need $N=15$ to have a good approximation of the return distribution in the case of $\gamma=0.8$.

\begin{table}
    \centering
    \begin{tabular}{cccc}
        \toprule
        $\gamma$ & $x_0$ & $c_0$ & $N_0$\\
        \midrule
         0.6& [1;1;1] & 0.5447  & 8 \\
         0.8& [1;1;1] & 0.5917  & 19 \\
         0.6& [6;6;6] & 1.7550  & 11 \\
         0.8& [6;6;6] & 2.6134  & 25 \\
         \bottomrule
    \end{tabular}
    \caption{ Constant $c_0$ in \eqref{eq:approximat:bound} and required number $N_0$ to obtain a good estimate for different values of $\gamma$ and $x_0$ in LQR, where $N_0$ is the smallest  integer such that $\sup_{z}|{F}^{K}_x(z)-{F}^{K}_{x,N_0}(z) | \leq c_0 \gamma^{N_0} \leq 0.01$.} 
    \label{tab:my_label}
\end{table}

Moreover, the value of the initial state $x_0$ has an influence on the shape of the return distribution. 
When $x_0$ is large, the random variable $w_k^T P A_K^{k+1}x_0$ dominates and, therefore, its distribution is close to a Gaussian distribution, as shown in Fig.~\ref{fig:approximation} (c) and (d).
If instead $x_0$ is  small, then the random variable $w_k^T P w_k$ plays a leading role, so the overall distribution is close to the chi-square one, as shown in Fig.~\ref{fig:approximation} (a) and (b). 

\begin{figure}[t]
    \centering
    \begin{subfigure}[b]{0.24\textwidth}
        \centering
        \includegraphics[width=\textwidth]{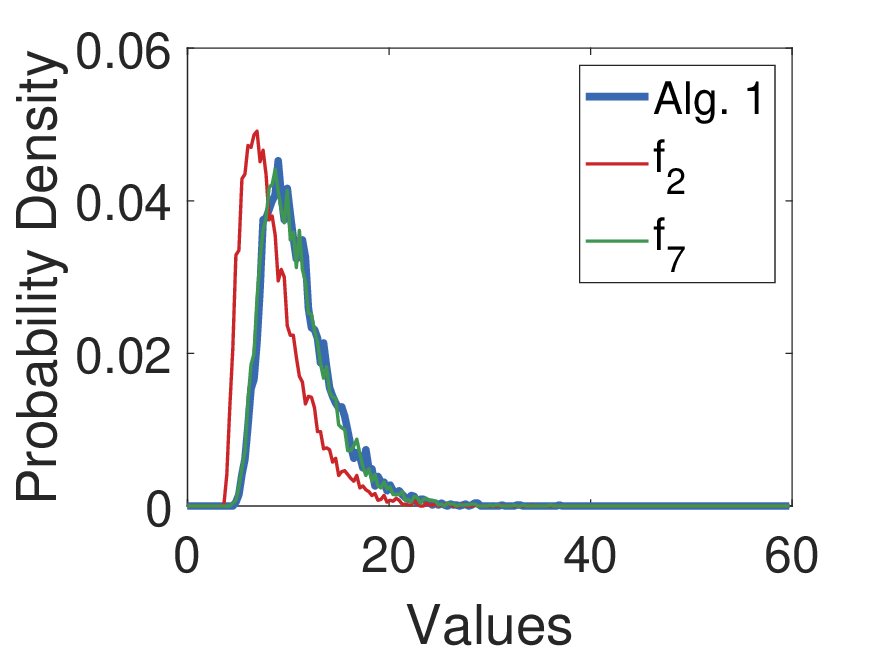}
        \caption{$\gamma=0.6$, $x_0=[1;1;1]$.}
        \label{fig:appro:f1}
    \end{subfigure}
    \hfill
    \begin{subfigure}[b]{0.24\textwidth}
        \centering
        \includegraphics[width=\textwidth]{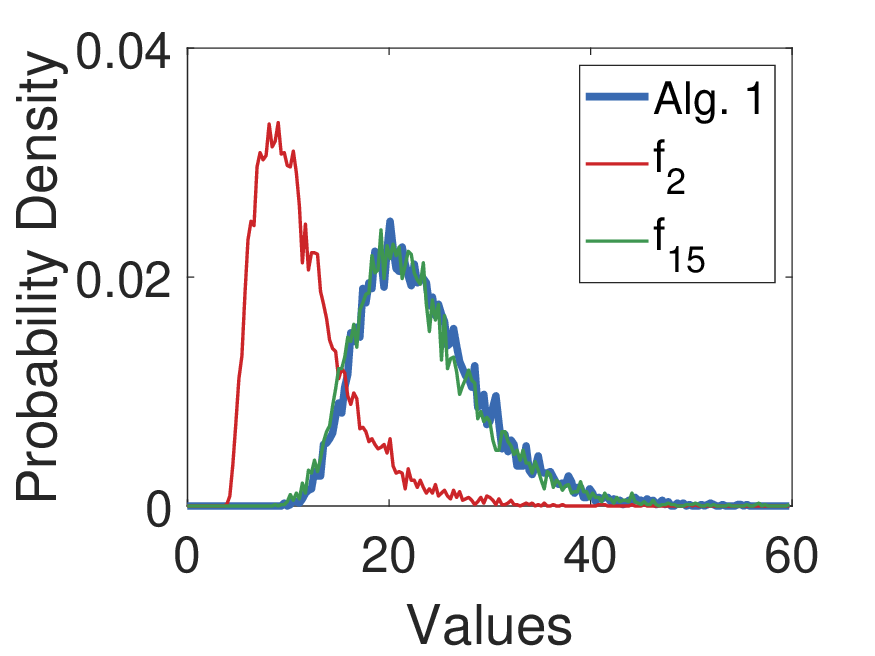}
        \caption{$\gamma=0.8$, $x_0=[1;1;1]$.}
        \label{fig:appro:f2}
    \end{subfigure}
    \hfill
    \begin{subfigure}[b]{0.24\textwidth}
        \centering
        \includegraphics[width=\textwidth]{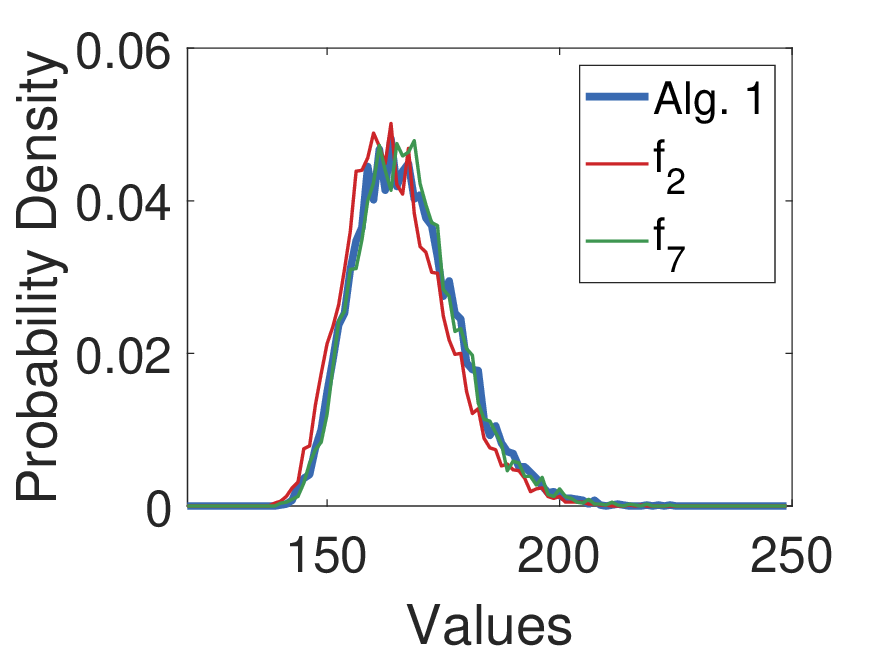}
        \caption{$\gamma=0.6$, $x_0=[6;6;6]$.}
        \label{fig:appro:f3}
    \end{subfigure}
    \hfill
    \begin{subfigure}[b]{0.24\textwidth}
        \centering
        \includegraphics[width=\textwidth]{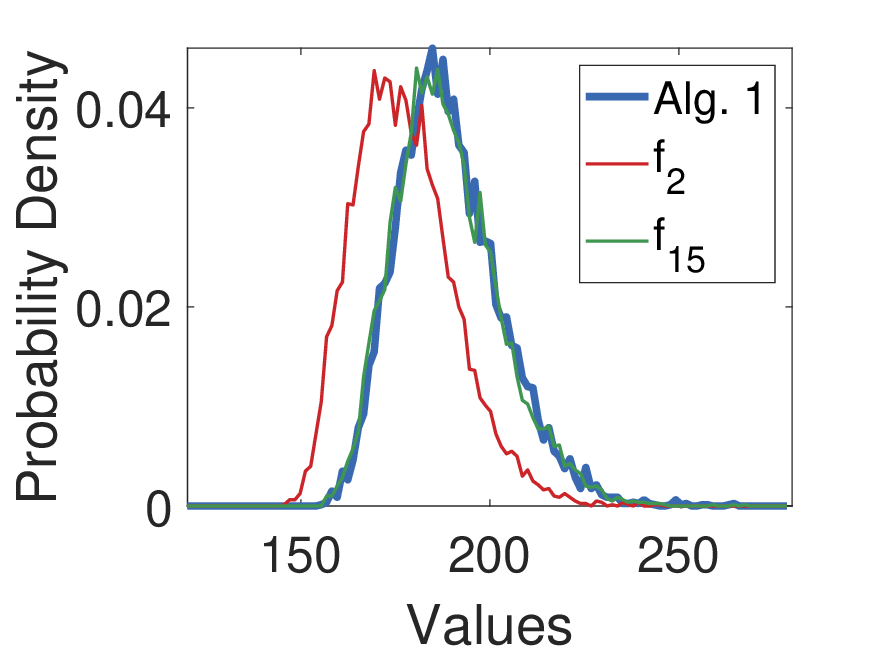}
        \caption{$\gamma=0.8$, $x_0=[6;6;6]$.}
        \label{fig:appro:f4}
    \end{subfigure}
    \caption{Return distribution and its approximation with finite number of random variables for different values of $\gamma$ and $x_0$ in LQR. Alg. 1 denotes the distribution returned by Algorithm~\ref{alg:algorithm_MFPE} and $f_N$ denotes the distribution of the approximated random return $G^K_N(x_0)$.} \label{fig:approximation}
\end{figure}

Next we perturb the matrices $A$, $B$ by an amount $\epsilon_A A$ and $\epsilon_B B$, respectively.  We select $x_0=[1;1;1]$. 
We compute the constants of $\tilde{c}_1$, $\tilde{c}_2$, the true sup difference between original and perturbed distributions, and the upper bounds in \eqref{eq:pert:bound}. The results are shown in Table~\ref{tab:2}. We observe that the perturbed distribution becomes significantly distinct from the original distribution when $\gamma$, $\epsilon_A$, and $\epsilon_B$ take on larger values. We also note that our computational upper bound becomes conservative when $\gamma$ is close to $1$.
The perturbed return distributions for different values of $\epsilon_A$ and $\epsilon_B$ are shown in Fig.~\ref{fig:perturbation}. We observe that large perturbations change the distributions dramatically.
\begin{figure}[t]
    \centering
    \begin{subfigure}[b]{0.24\textwidth}
        \centering
        \includegraphics[width=\textwidth]{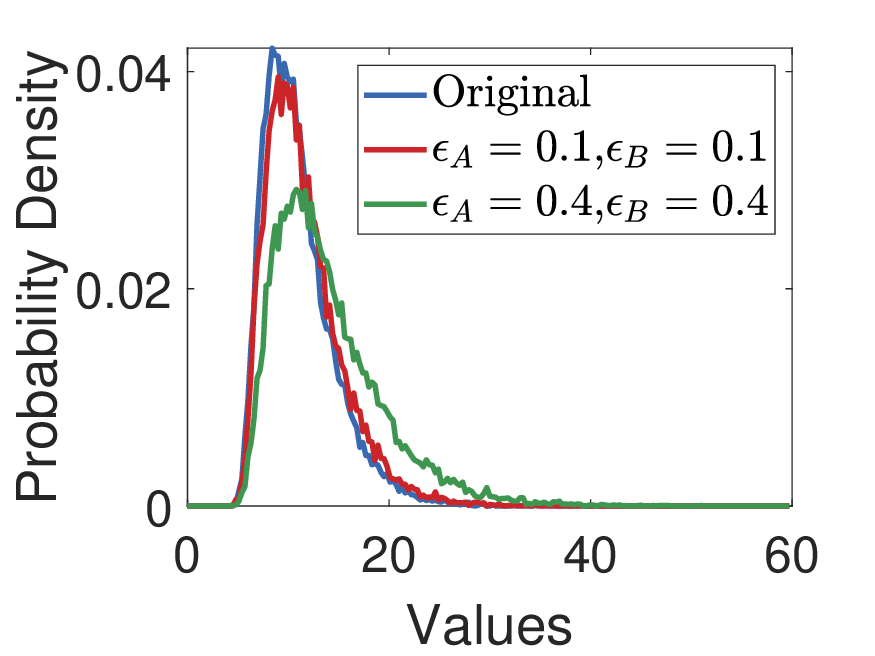}
        \caption{$\gamma=0.6$.}
        \label{fig:pert:f1}
    \end{subfigure}
    \hfill
    \begin{subfigure}[b]{0.24\textwidth}
        \centering
        \includegraphics[width=\textwidth]{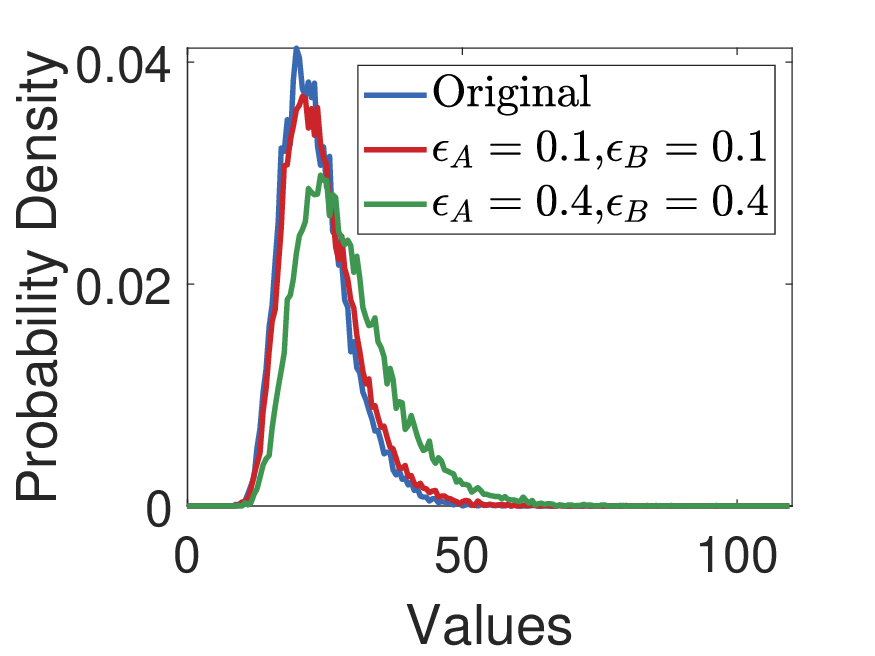}
        \caption{$\gamma=0.8$.}
        \label{fig:pert:f2}
    \end{subfigure}
    \caption{Original and perturbed return distributions for different values of $\gamma$, $\epsilon_A$ and $\epsilon_B$ in LQR.} \label{fig:perturbation}
\end{figure}

\begin{table}
    \centering
    \begin{tabular}{ccccccc}
        \toprule
        $\gamma$ & $\epsilon_A$ & $\epsilon_B$ & $\tilde{c}_1$ & $\tilde{c}_2$ & Sup difference & UB   \\
        \midrule
         0.6& 0.1 & 0.1 & 6.5  & 4.6 & 0.051 & 0.33\\
         0.6& 0.4 & 0.4 & 20.3  & 11.9 & 0.24 & 0.52\\
         0.8& 0.1 & 0.1 & 12.4  & 9.9 & 0.056 &  0.53\\
         0.8& 0.4 & 0.4 & 30.5  & 20.9 & 0.26 &  0.80\\
         \bottomrule
    \end{tabular}
    \caption{ Computation of the actual maximal difference between the perturbed and the original distributions, and computational upper bound (UB) for different values of $\gamma$, $\epsilon_A$ and $\epsilon_B$ in LQR. The constants $\tilde{c}_1$ and $\tilde{c}_2$ are those in \eqref{eq:pert:bound}. Sup difference is the value of $\sup_{z}|{F}^{K}_x(z)-\tilde{F}^{K}_{x}(z) |$ while UB is the value of $\tilde{c}_1 \left\|\Delta A_K \right\| + \tilde{c}_2 \left\| \Delta A_K \right\|^2$. }
    \label{tab:2}
\end{table}

\subsection{LQG}
In this section, we assume that the system is partially observable and we have the observation $y_t = C x_t + s_t$, where $C=[1,0,0;0,1,0]$. We assume that the disturbance $s_t$ is normally distributed with zero mean. 
We design the state estimator and controller
\begin{align}
    & \hat{x}_{t+1} = A \hat{x}_t + B u_t + L(y_t - C  \hat{x}_t), \nonumber \\
    & u_t = K \hat{x}_t. \nonumber
\end{align}
where the controller is selected the same as that in LQR and the observer is selected as $L=[0.21,0.01;0.01,0.32;0,2.32]$.
We set $x_0 = [1;1;1]$, $\hat{x}_0=[0;0;0]$. The simulation results for LQG are presented in Fig.~\ref{fig:LQG}.
Similarly, we denote by $f_N$ the distribution of the approximated random return $G^{KL}_N(\bar{x})$ in \eqref{eq:dist_func:appro:LQG} obtained based on $N$ random variables.
We use the Monte Carlo (MC) method with sufficiently many data to construct an empirical distribution that serves as the baseline distribution for comparison. 
As shown in Fig.~\ref{fig:LQG}, when $\gamma=0.6$, we need $N=8$ number of random variables to obtain a good approximation of the return distribution. When $\gamma = 0.8$, a greater number $N=17$ is needed to achieve reliable approximation of the return distribution.
\begin{figure}[t]
    \centering
    \begin{subfigure}[b]{0.24\textwidth}
        \centering
        \includegraphics[width=\textwidth]{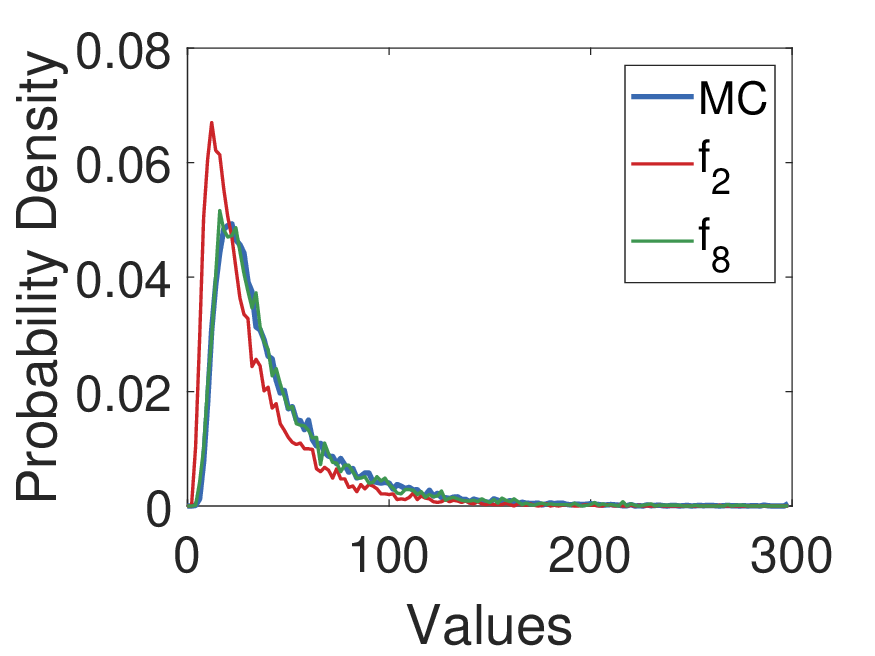}
        \caption{$\gamma=0.6$.}
        \label{fig:lqg:f1}
    \end{subfigure}
    \hfill
    \begin{subfigure}[b]{0.24\textwidth}
        \centering
        \includegraphics[width=\textwidth]{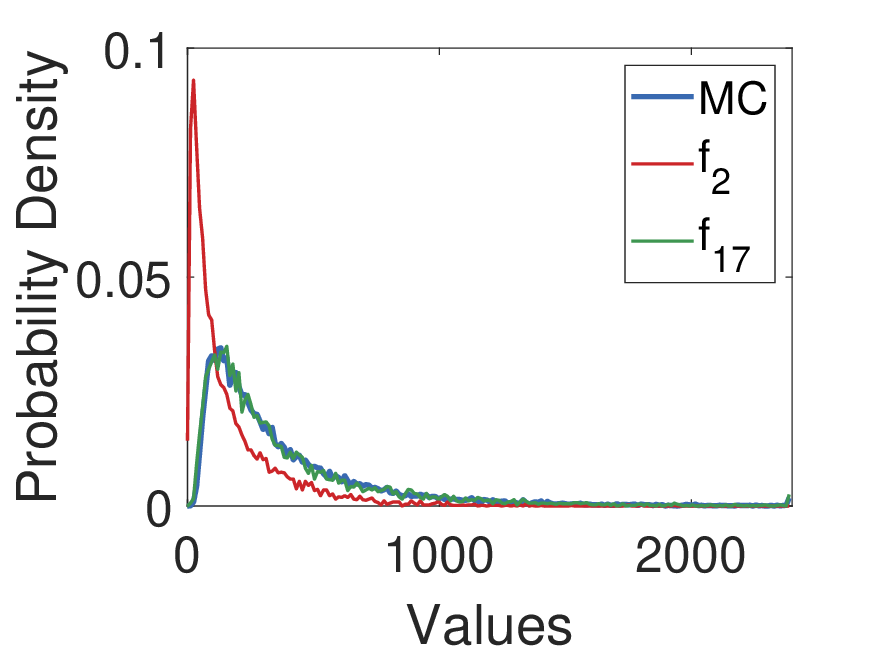}
        \caption{$\gamma=0.8$.}
        \label{fig:lqg:f2}
    \end{subfigure}
    \caption{Return distribution and its approximation with finite number of random variables for different values of $\gamma$ in LQG. MC denotes the distribution estimated using the Monte Carlo method and $f_N$ denotes the distribution of the approximated random return $G^{KL}_N(\bar{x})$.} \label{fig:LQG}
\end{figure}

\section{Conclusions}\label{sec:conclusion}
We have proposed a new distributional approach to the classic discounted LQR problem. Specifically, we have first provided an analytic expression for the exact random return  that depends on infinitely many random variables. 
In this context, we have shown that the variance remains bounded if the fourth moment of the disturbance is bounded. Furthermore, we have conducted an analysis of distribution sensitivity.  
Besides, we have proposed a model-free method for evaluating the return distribution, with theoretical analysis of its sample complexity.
Since the computation of this expression is difficult in practice, we have also proposed an approximate expression for the distribution of the random return that only depends on a finite number of random variables, and have further characterised the approximation error. 
Moreover, we have extended most of the above results for LQR to the partially observable case.

This work provides a framework for distributional LQR: it inherits the advantages of DRL methods compared to standard RL ones that rely on the expected return to evaluate a given policy, but it also provides an analytic expression for the return distribution, an aspect where current DRL methods significantly lack. 
Our framework provides richer information for linear control systems, i.e., the whole distribution of the random return, and enables us to consider more general objectives, e.g., risk-averse control. 
Future research includes exploring policy improvement for risk-averse control using the learned return distribution.

\appendix\label{sec:appendix}
\subsection{Proof of Theorem~\ref{Thm:exact:dist}}\label{app:thm:exact:dist}
Recall that $x' = A_K x + v_0$, where $v_0$ is a  random variable sampled from the distribution $\mathcal{D}$ and is independent from $w_k$, $k\in \mathbb{N}$, in \eqref{eq:dist_func}. 
Substituting \eqref{eq:dist_func} into the right hand side of the equation \eqref{eq:rv:bellman}, we have that
\begin{align*}
    & x^{\rm{T}}(Q+K^{\rm{T}} R K)x + \gamma G^{K}(x') \\
    =  & x^{\rm{T}}Q_Kx + \gamma x'^{\rm{T}} P x' +  \sum_{t=0}^{\infty} \gamma^{t+2} w_t^{\rm{T}} P w_t \nonumber \\
    & + 2 \sum_{t=0}^{\infty} \gamma^{t+2} w_t^{\rm{T}}  P A_K^{t+1}  x' \\
    &+ 2  \sum_{t=1}^{\infty} \gamma^{t+2} w_t^{\rm{T}} P A_K \sum_{i=0}^{t-1} A_K^{t-1-i} w_i  \\
    = & x^{\rm{T}}Q_Kx + \gamma (A_K x + v_0)^{\rm{T}} P (A_K x + v_0) \\
    &+ \gamma^2 \sum_{t=0}^{\infty} \gamma^t w_t^{\rm{T}} P w_t  + 2 \gamma^2 \sum_{t=1}^{\infty} \gamma^t w_t^{\rm{T}} P \sum_{i=0}^{t-1} A_K^{t-i} w_i \\
    &+ 2\gamma^2 \sum_{t=0}^{\infty} \gamma^t w_t^{\rm{T}}  P A_K^{t+1}  (A_K x + v_0) \\
    = & x^{\rm{T}}(Q_K + \gamma A_K^{\rm{T}} P A_K)x + \underbrace{\gamma v_0^{\rm{T}} P v_0 + \gamma^2 \sum_{t=0}^{\infty} \gamma^t w_t^{\rm{T}} P w_t }_{:=T_1} \\
    & +\underbrace{ 2\gamma v_0^{\rm{T}} PA_K x + 2\gamma^2 \sum_{t=0}^{\infty} \gamma^t w_t^{\rm{T}}  P A_K^{t+2} x}_{:=T_2} \\
    & + \underbrace{2 \gamma^2 \sum_{t=1}^{\infty} \gamma^t w_t^{\rm{T}} P  \sum_{i=0}^{t-1} A_K^{t-i} w_i + 2\gamma^2 \sum_{t=0}^{\infty} \gamma^t w_t^{\rm{T}}  P  A_K^{t+1} v_0}_{:=T_3}.
\end{align*}
Define $ \xi_0 :=v_0$, $\xi_t = w_{t-1}$, $t=1,2,\ldots$. 
From the definition of the term $T_1$, we have that
\begin{align*}
    T_1 & = \gamma v_0^{\rm{T}} P v_0 + \gamma^2 \sum_{t=0}^{\infty} \gamma^t w_t^{\rm{T}} P w_t \\
     & \mathop{=}^{k=t+1}  \gamma \xi_0^{\rm{T}} P \xi_0 + \gamma \sum_{k=1}^{\infty} \gamma^k \xi_{k}^{\rm{T}} P \xi_{k} 
    =  \gamma \sum_{k=0}^{\infty} \gamma^k \xi_{k}^{\rm{T}} P \xi_{k}.
\end{align*}
For the term $T_2$, we have that
\begin{align*}
    T_2 & = 2\gamma v_0^{\rm{T}} PA_K x + 2\gamma^2 \sum_{t=0}^{\infty} \gamma^t w_t^{\rm{T}}  P A_K^{t+2} x \\
    &= 2\gamma \xi_0^{\rm{T}} P A_K x + 2\gamma^2 \sum_{t=0}^{\infty} \gamma^t \xi_{t+1}^{\rm{T}}  P A_K^{t+2} x \\
    &\mathop{=}^{k=t+1}  2\gamma \xi_0^{\rm{T}} P A_K x + 2\gamma \sum_{k=1}^{\infty} \gamma^k \xi_{k}^{\rm{T}}  P A_K^{k+1} x \\
    & =  2\gamma  \sum_{k=0}^{\infty} \gamma^k \xi_k^{\rm{T}} P A_K^{k+1} x.
\end{align*}
Using similar techniques for the next term, we obtain that
$T_3 = 2 \gamma \sum_{k=1}^{\infty} \gamma^{k} \xi_{k}^{\rm{T}} P A_K  \sum_{i=0}^{k-1} A_K^{k-1-i} \xi_i.$ 
Due to the fact that $P =Q+ K^{\rm{T}} R K + \gamma A_K^{\rm{T}} P A_K$, we have 
\begin{align}\label{eq:proof:P1:temp1}
    & x^{\rm{T}}Q_Kx + \gamma G^{K}(x') 
    =  x^{\rm{T}} P x + T_1 +T_2 +T_3 \nonumber \\
    = & x^{\rm{T}} P x + \gamma \sum_{k=0}^{\infty} \gamma^k \xi_{k}^{\rm{T}} P \xi_{k} + 2\gamma  \sum_{k=0}^{\infty} \gamma^k x^{\rm{T}} P A_K^{k+1} \xi_k \nonumber \\
    & + 2 \gamma \sum_{k=1}^{\infty} \gamma^{k} \xi_{k}^{\rm{T}} P A_K  \sum_{i=0}^{k-1} A_K^{k-1-i} \xi_i,
\end{align}
which is in the same form as in \eqref{eq:dist_func}.
Since $\{ \xi_k\}_{k=0}^{\infty}$ and $\{ w_k\}_{k=0}^{\infty}$ are i.i.d.,
we have that the two random variables \eqref{eq:dist_func} and \eqref{eq:proof:P1:temp1} have the same distribution, i.e.,
$G^{K}(x)  \mathop{=}\limits^{D}  x^{\rm{T}}Q_Kx + \gamma G^{K}(x').$

\subsection{Proof of Theorem \ref{thm:variance}}\label{app:thm:variance}

By virtue of Jensen's Inequality, we have $\mathbb{E}^2[\left\| w_k\right\|^2] \leq \mathbb{E}[\left\| w_k \right\|^4]$ and $\mathbb{E}^2[\left\| w_k\right\|] \leq \mathbb{E}[\left\| w_k \right\|^2]$. Therefore, we have $\mathbb{E}[\left\| w_k\right\|^2]\leq \sigma_4^2$ and $\mathbb{E}[\left\| w_k\right\|]\leq \sigma_4$.
Since $(a+b+c+d)^2 \leq 4 (a^2 +b^2 +c^2 +d^2)$, we have
\begin{align}\label{eq:pf:var:1}
    &\mathbb{E}\Big[G^K(x) G^K(x)\Big] \nonumber \\
    &\leq 4\mathbb{E}\Big[  (x^{\rm{T}} P x)^2 + \big(\sum_{k = 0 }^{\infty}  \gamma^{k+1}  w_k^{\rm{T}} P w_k\big)^2 \nonumber \\
    &\quad +  \big(2 \sum_{k = 0 }^{\infty} \gamma^{k+1} w_k^{\rm{T}} P A_K^{k+1}x\big)^2  \nonumber \\
    &\quad + \big( 2  \sum_{k = 1 }^{\infty} \gamma^{k+1} w_k^{\rm{T}} P \sum_{\tau=0}^{k-1} A_K^{k-\tau}w_{\tau}\big)^2 \Big].
\end{align}
We handle the terms one by one. The first term can be easily bounded by 
\begin{align}\label{eq:pf:var:2}
    (x^{\rm{T}} P x)^2 \leq \left\|x\right\|^4 \left\| P \right\|^2.
\end{align}
By virtue of the Cauchy Product that $\Big(\sum_{k=0}^{\infty} a_k \Big)^2 = \sum_{k=0}^{\infty} \sum_{l=0}^k a_l a_{k-l}$, the second term can be bounded by 
\begin{align}\label{eq:pf:var:3}
    &\mathbb{E} \Big[ \big(\sum_{k = 0 }^{\infty}  \gamma^{k+1}  w_k^{\rm{T}} P w_k\big)^2 \Big] \nonumber \\
    & \leq \mathbb{E} \Big[ \big(\sum_{k = 0 }^{\infty}  \gamma^{k+1} \left\| w_k\right\|^2 \left\| P \right\|\big)^2 \Big] \nonumber \\
    & = \left\| P \right\|^2 \mathbb{E} \Big[ \sum_{k = 0 }^{\infty} \gamma^{k+2} \sum_{l=0}^k \left\| w_l\right\|^2 \left\| w_{k-l}\right\|^2 \Big] \nonumber \\
     & = \left\| P \right\|^2  \sum_{k = 0 }^{\infty} \gamma^{k+2} \sum_{l=0}^k \mathbb{E} \Big[\left\| w_l\right\|^2 \left\| w_{k-l}\right\|^2\Big]  \nonumber \\
    &\leq \left\| P \right\|^2 \sum_{k = 0 }^{\infty} \gamma^{k+2} (k+1) \sigma_4^4 \nonumber \\
    &= \sigma_4^4  \left\| P \right\|^2 \frac{\gamma^2}{(1-\gamma)^2}.
\end{align}
The first inequality holds since $w_kPw_k \geq 0$.  The second inequality holds since $\mathbb{E}[\left\| w_l\right\|^2 \left\| w_{k-l}\right\|^2] = \mathbb{E}[\left\| w_l\right\|^2] \mathbb{E}[\left\| w_{k-l}\right\|^2] \leq \sigma_4^4$ when $l\neq k-l$ and $\mathbb{E}[\left\| w_l\right\|^2 \left\| w_{k-l}\right\|^2] = \mathbb{E}[\left\| w_l\right\|^4] \leq \sigma_4^4$ when $l=k-l$. 
The last equality holds since $\sum_{k = 0 }^{\infty} \gamma^{k+2} (k+1) = \frac{\gamma^2}{(1-\gamma)^2}$.

For the third term, we have
\begin{align}\label{eq:pf:var:4}
    &\mathbb{E}\Big[ \big(2 \sum_{k = 0 }^{\infty} \gamma^{k+1} w_k^{\rm{T}} P A_K^{k+1}x\big)^2\Big] \nonumber \\
    & = 4 \mathbb{E}\Big[ \sum_{k = 0 }^{\infty} \sum_{l=0}^k \gamma^{l+1} w_l^{\rm{T}} P A_K^{l+1}x \gamma^{k-l+1} w_{k-l}^{\rm{T}} P A_K^{k-l+1}x \Big] \nonumber \\
    & \leq 4 \left\| P\right\|^2 \left\| x\right\|^2 \mathbb{E}\Big[\sum_{k = 0 }^{\infty} \gamma^{k+2} \sum_{l=0}^k \left\| w_l \right\| \left\| A_K^{l+1}\right\| \left\|w_{k-l}\right\| \nonumber \\
    & \quad  \times \left\| A_K^{k-l+1}\right\| \Big]\nonumber \\
    &\leq 4 \left\| P\right\|^2 \left\| x\right\|^2 \sum_{k = 0 }^{\infty} (\gamma\rho)^{k+2}  \sum_{l=0}^k \mathbb{E}\Big[ \left\| w_l \right\| \left\|w_{k-l}\right\| \Big] \nonumber \\
    &\leq 4 \left\| P\right\|^2 \left\| x\right\|^2 \sigma_4^2  \sum_{k = 0 }^{\infty} (k+1) (\gamma \rho_K)^{k+2} \nonumber \\
    &= 4 \left\| P\right\|^2 \left\| x\right\|^2 \sigma_4^2 \frac{\gamma^2 \rho_K^2}{(1-\gamma \rho_K)^2}.
\end{align}
The first equality follows from the Cauchy Product. The first inequality follows from  $w_l^{\rm{T}} P A_K^{l} x \leq \left\| P\right\| \left\|  A_K^{l} \right\| \left\| w_l\right\| \left\| x\right\|$ and $w_{k-l}^{\rm{T}} P A_K^{k-l+1} x \leq \left\| P\right\| \left\|  A_K^{k-l+1} \right\| \left\| w_{k-l}\right\| \left\| x\right\|$. 
The second inequality holds since $\left\| A_K^{l} \right\| \leq \left\| A_K\right\|^{l}\leq \rho_K^{l}$. The third inequality holds since $\mathbb{E}[ \left\| w_l \right\| \left\|w_{k-l}\right\|] \leq \sigma_4^2$ when $l=k-l$ and $l\neq k-l$. The last equality holds since $\sum_{k = 0 }^{\infty} (\gamma\rho)^{k+2} (k+1) = \frac{\gamma^2\rho^2}{(1-\gamma\rho)^2}$. 

For the fourth term, by virtue of the Cauchy Product, we have
\begin{align}\label{eq:pf:var:5}
    &\mathbb{E}\Big[ \big( 2  \sum_{k = 1 }^{\infty} \gamma^{k+1} w_k^{\rm{T}} P \sum_{\tau=0}^{k-1} A_K^{k-\tau}w_{\tau}\big)^2\Big] \nonumber \\
      &= \mathbb{E}\Big[ \big( 2  \sum_{k = 0 }^{\infty} \gamma^{k+2} w_{k+1}^{\rm{T}} P \sum_{\tau=0}^{k} A_K^{k+1-\tau}w_{\tau}\big)^2\Big] \nonumber \\
    &=4 \mathbb{E}\Big[ \sum_{k=0}^{\infty} \gamma^{k+4} \sum_{l=0}^k w_{l+1}^{\rm{T}} P(\sum_{\tau=0}^{l} A_K^{l+1-\tau} w_{\tau}) \nonumber \\
    &\quad \times w_{k-l+1}^{\rm{T}} P (\sum_{\tau=0}^{k-l} A_K^{k-l+1-\tau} w_{\tau}) \Big]. 
\end{align}

Let $\xi :=  w_{l+1}^{\rm{T}} P(\sum_{\tau=0}^{l} A_K^{l+1-\tau} w_{\tau})  \times w_{k-l+1}^{\rm{T}} P (\sum_{\tau=0}^{k-l} A_K^{k-l+1-\tau} w_{\tau})$. Recall that  the random variables $w_k$ are independent from each other  and $\mathbb{E}[w_{k}]=0$ for all $k\in\mathbb{N}$. It yields that  when $l>k-l$, $\mathbb{E}[\xi]=0$, and when $l<k-l$,  $\mathbb{E}[\xi]=0$.  Thus, \eqref{eq:pf:var:5} can be simplified to be with the items when $k=2l$, i.e., 
\begin{align}\label{eq:pf:var:6}
    &\mathbb{E}\Big[ \big( 2  \sum_{k = 1 }^{\infty} \gamma^{k+1} w_k^{\rm{T}} P \sum_{\tau=0}^{k-1} A_K^{k-\tau}w_{\tau}\big)^2\Big] \nonumber \\
    & =4 \mathbb{E}\Big[ \sum_{k=0}^{\infty} \gamma^{k+4} \sum_{l=0}^k w_{l+1}^{\rm{T}} P(\sum_{\tau=0}^{l} A_K^{l+1-\tau} w_{\tau}) \nonumber \\
    &\quad \times w_{k-l+1}^{\rm{T}} P (\sum_{\tau=0}^{k-l} A_K^{k-l+1-\tau} w_{\tau}) \Big] \nonumber \\
     & =4 \mathbb{E}\Big[ \sum_{l=0}^{\infty} \gamma^{2l+4} \big(w_{l+1}^{\rm{T}} P(\sum_{\tau=0}^{l} A_K^{l+1-\tau} w_{\tau}) \big)^2\Big] \nonumber \\ 
    & \leq 4 \left\| P\right\|^2 \sum_{l=0}^{\infty} \gamma^{2l+4} \mathbb{E}\Big[ \left\| w_{l+1}\right\|^2\Big]   \mathbb{E}\Big[\left\| \sum_{\tau=0}^{l} A_K^{l+1-\tau} w_{\tau}\right\|^2 \Big] \nonumber \\
     & \leq 4 \left\| P\right\|^2 \sigma_4^2 \sum_{l=0}^{\infty} \gamma^{2l+4}  \mathbb{E}\Big[ \left\| \sum_{\tau=0}^{l} A_K^{l+1-\tau} w_{\tau}\right\|^2 \Big] \nonumber \\
          & \leq 4 \left\| P\right\|^2 \sigma_4^2 \sum_{l=0}^{\infty} \gamma^{2l+4}  \mathbb{E}\Big[\big( \sum_{\tau=0}^{l} w_{\tau}^{\rm{T}}(A_K^{l+1-\tau})^{\rm{T}} \big) \nonumber \\ 
          &\quad \times  \big( \sum_{\tau=0}^{l} A_K^{l+1-\tau} w_{\tau}\big) \Big] \nonumber \\ 
           & = 4 \left\| P\right\|^2 \sigma_4^2 \mathbb{E}\Big[\sum_{l=0}^{\infty} \gamma^{2l+4}  \sum_{\tau=0}^{l} \sum_{\kappa=0}^{l}\big( w_{\tau}^{\rm{T}}(A_K^{l+1-\tau})^{\rm{T}}  A_K^{l+1-\kappa} w_{\kappa}\big) \Big] \nonumber \\ 
           & = 4 \left\| P\right\|^2 \sigma_4^2 \sum_{l=0}^{\infty} \gamma^{2l+4}  \sum_{\tau=0}^{l} \sum_{\kappa=0}^{l}\mathbb{E}\Big[\big( w_{\tau}^{\rm{T}}(A_K^{l+1-\tau})^{\rm{T}}  A_K^{l+1-\kappa} w_{\kappa}\big) \Big] \nonumber \\ 
              & = 4 \left\| P\right\|^2 \sigma_4^2 \sum_{l=0}^{\infty} \gamma^{2l+4}  \sum_{\tau=0}^{l} \mathbb{E}\Big[\big( w_{\tau}^{\rm{T}}(A_K^{l+1-\tau})^{\rm{T}}  \times A_K^{l+1-\tau} w_{\tau}\big) \Big] \nonumber \\ 
           & \leq  4 \left\| P\right\|^2 \sigma_4^2 \sum_{l=0}^{\infty} \gamma^{2l+4}  \sum_{\tau=0}^{l} \rho_K^{2(l+1-\tau)}\mathbb{E} \Big[ \left\|w_{\tau}\right\|^2 \Big] \nonumber \\ 
       &\leq 4 \left\| P\right\|^2 \sigma_4^4 \sum_{l=0}^{\infty} \gamma^{2l+4}  \sum_{\tau=0}^{l} \rho_K^{2(l+1-\tau)} \nonumber \\ 
    &\leq 4 \left\| P\right\|^2 \sigma_4^4 \sum_{l=0}^{\infty} \gamma^{2l+4} \frac{\rho_K^2}{1-\rho_K^2} \nonumber \\
    &\leq \frac{4 \left\| P\right\|^2 \sigma_4^4 \rho_K^2 \gamma^4}{(1-\rho_K^2)(1-\gamma^2)}.
\end{align}
The first inequality holds since $w_{l+1}$ and $w_{\tau}$ are independent for all $\tau=0,\ldots,l$. The last equality holds since 
$w_{\tau}$ and $w_{\kappa}$ are independent for $\tau \neq \kappa$ and $\mathbb{E}[w_{\tau}]=0$ for all $\tau=0,\ldots,l$. 
The second to last inequality holds since $\sum_{\tau=0}^{l} \rho_K^{2(l+1-\tau)} \leq \frac{\rho_K^2}{1-\rho_K^2}$.

Combining \eqref{eq:pf:var:2}, \eqref{eq:pf:var:3}, \eqref{eq:pf:var:4}, \eqref{eq:pf:var:6} and \eqref{eq:pf:var:1}, we have
\begin{align*}
    &\mathbb{E}\Big[G^K(x) G^K(x)\Big] \nonumber \\
    &\leq 4\left\|x\right\|^4 \left\| P \right\|^2 + \frac{4 \sigma_4^4  \left\| P \right\|^2 \gamma^2}{(1-\gamma)^2} +  \frac{16 \left\| P\right\|^2 \left\| x\right\|^2 \sigma_4^2\gamma^2 \rho_K^2}{(1-\gamma \rho_K)^2} \nonumber \\
    &\quad +\frac{16 \left\| P\right\|^2 \sigma_4^4 \rho_K^2 \gamma^4}{(1-\rho_K^2)(1-\gamma^2)}.
\end{align*}
Since $ \mathbb{E}\big[G^K(x)\big]$ is bounded, the variance of $G^K(x)$ is bounded.
The proof is complete.

\subsection{Proof of Theorem~\ref{thm:perturb}}\label{app:thm:perturb}
Before analyzing the effect of perturbations on the return distribution, it is necessary to investigate how perturbations affect the solution to the Lyapunov equation. The following lemma presents the well-known sensitivity result for LQR.
\begin{lemma}\label{lemma:perturbation_bound} \cite{laub1990sensitivity}
Let $X$ be the unique solution of the Lyapunov equation $X = Q +A^{\rm{T}} X A$ for a stable matrix $A$. 
Let $\tilde{X}$ be the unique solution of the perturbed Lyapunov equation $\tilde{X}=Q+\tilde{A}^{\rm{T}} \tilde{X} \tilde{A}$ for a stable matrix $\tilde{A}=A+ \Delta A$. Then, when $l_0> 2\epsilon_0$, where $l_0= \left\| H_0^{-1}\right\|^{-1}$, $H_0 = I \otimes I - A^{\rm{T}} \otimes A^{\rm{T}}$, and $\epsilon_0 = \left\|A \right\|_F \left\| \Delta A \right\|_F + \frac{1}{2} \left\| \Delta A \right\|_F^2 $,
we have 
\begin{align*}
    \left\| X - \tilde{X}\right\|_F \leq \frac{2  \left\| X\right\|_F \epsilon_0 }{l_0-2\epsilon_0}.
\end{align*}
\end{lemma}
Lemma~\ref{lemma:perturbation_bound} analyzes the sensitivity of the Lyapunov equation for the canonical form of LQR. 
For discounted LQR, the sensitivity analysis of the Lyapunov equation is presented in the following lemma.
\begin{lemma}
Let $l= \left\| H^{-1}\right\|^{-1}$, $H = I \otimes I - \gamma A_K^{\rm{T}} \otimes A_K^{\rm{T}}$, $\epsilon =  \gamma \left\| A_K \right\|_F \left\| \Delta A_K \right\|_F + \frac{\gamma}{2} \left\| \Delta A_K \right\|_F^2$. 
 If $l>2\epsilon$,  we have 
\begin{align}\label{eq:perturb:bound:P}
    \left\| P - \tilde{P}\right\| \leq \left\| P - \tilde{P}\right\|_F \leq \frac{2  \left\| P \right\|_F \epsilon}{l - 2 \epsilon}. 
\end{align}
\end{lemma}
\begin{proof}
It directly follows from Lemma~\ref{lemma:perturbation_bound} and 
 $\left\|M\right\|_2 \leq \left\|M\right\|_F \leq \sqrt{n}\left\|M\right\|_2$, for any matrix $M\in \mathbb{R}^{n \times n}$. 
\end{proof}

Back to the sensitivity of perturbations on the return distribution, we define $\tilde{Y}:= G^K(x) - \tilde{G}^{K}(x)$, we have
\begin{align}\label{eq:pf:pert:1}
     &\sup_{z}|{F}^{K}_x(z)-\tilde{F}^{K}_{x}(z) | \nonumber \\
     &=\sup_{z}|\mathbb{P}(\tilde{G}^{K}(x) \leq z) -\mathbb{P}(G^K(x)\leq z) | \nonumber \\
     &= \sup_{z}|\mathbb{P}( \tilde{G}^{K}(x) \leq z) -\mathbb{P}( \tilde{G}^{K}(x) +\tilde{Y}\leq z) | \nonumber \\
     &= \sup_{z}\Big|\mathbb{P}(\tilde{G}^{K}(x)\leq z) \int_{-\infty}^{\infty} \mathbb{P}(\tilde{Y} = t)dt \nonumber \\
     &\quad -\int_{-\infty}^{\infty} \mathbb{P}( \tilde{G}^{K}(x) \leq z-t) \mathbb{P}(\tilde{Y} = t) dt \Big| \nonumber \\
     &= \sup_{z}\Big| \int_{-\infty}^{\infty} \mathbb{P}(\tilde{Y} = t) \big(  \tilde{F}^K_{x}(z) - \tilde{F}^K_{x}(z-t) \big) dt \Big| \nonumber \\
     &\leq \sup_{z}\Big| \int_{-\infty}^{\infty} \mathbb{P}(\tilde{Y} = t) \tilde{f}_{\max} |t| dt \Big| \nonumber \\
     & = \tilde{f}_{\max} \mathbb{E}\big[ |\tilde{Y}|\big],
\end{align}
where $\tilde{f}_{\max}$ is an upper bound of the PDF of $\tilde{G}^{K}(x)$ and the last inequality follows from the mean value theorem.

From the definition of $\tilde{Y}$, we have 
\begin{align}\label{eq:pf:pert:2}
    &\mathbb{E}|\tilde{Y}| = \mathbb{E}\Big[ \Big|x^{\rm{T}} (P - \tilde{P}) x +  \sum_{k = 0}^{\infty}  \gamma^{k+1}  w_k^{\rm{T}} (P-\tilde{P}) w_k \nonumber \\
    &\quad +  2 \sum_{k = 0 }^{\infty} \gamma^{k+1} w_k^{\rm{T}} ( P A_K^{k+1} - \tilde{P} \tilde{A}_K^{k+1} )x \nonumber \\
    &\quad  + 2  \sum_{k = 1 }^{\infty} \gamma^{k+1} w_k^{\rm{T}} \big(P \sum_{\tau=0}^{k-1} {A}_K^{k-\tau}w_{\tau}-  \tilde{P} \sum_{\tau=0}^{k-1} \tilde{A}_K^{k-\tau}w_{\tau}\big) \Big|\Big] \nonumber \\
    &\leq |x^{\rm{T}} ( \tilde{P} - P ) x| + \sum_{k = 0}^{\infty}  \gamma^{k+1}\mathbb{E} \Big[ \Big|   w_k^{\rm{T}} (\tilde{P}- P) w_k\Big| \Big]\nonumber \\
    &\quad +  2 \sum_{k = 0 }^{\infty} \gamma^{k+1} \mathbb{E}\Big[ \Big| w_k^{\rm{T}} (  \tilde{P} \tilde{A}_K^{k+1} -P A_K^{k+1} )x \Big| \Big]\nonumber \\
    &\quad  + 2  \sum_{k = 1 }^{\infty} \gamma^{k+1} \mathbb{E}\Big[ \Big| w_k^{\rm{T}} \big( \tilde{P} \sum_{\tau=0}^{k-1} \tilde{A}_K^{k-\tau}- P \sum_{\tau=0}^{k-1} {A}_K^{k-\tau}   \big)w_{\tau} \Big| \Big].
\end{align}
We handle the terms in the above inequality one by one.
By virtue of the Holder's inequality, the first term can be bounded by 
\begin{align}\label{eq:pf:pert:3}
    |x^{\rm{T}} ( \tilde{P} - P ) x| \leq \left\|x\right\|^2 \left\| \tilde{P} - P\right\|.
\end{align}
Similarly, the second term can be bounded by
\begin{align}\label{eq:pf:pert:4}
    &\sum_{k = 0}^{\infty}  \gamma^{k+1}\mathbb{E} \Big[\Big|   w_k^{\rm{T}} (\tilde{P}- P) w_k\Big|\Big] \nonumber \\
    &\leq \sum_{k = 0}^{\infty}  \gamma^{k+1}  \sigma^2 \left\| \tilde{P} - P\right\| 
    \leq  \frac{\sigma^2 \gamma }{1-\gamma} \left\| \tilde{P} - P\right\|.
\end{align}
For the third term, we have 
\begin{align}\label{eq:pf:pert:5}
    &2 \sum_{k = 0 }^{\infty} \gamma^{k+1} \mathbb{E}\Big[\Big| w_k^{\rm{T}} (  \tilde{P} \tilde{A}_K^{k+1} -P A_K^{k+1} )x \Big| \Big]\nonumber \\
    &= 2 \sum_{k = 0 }^{\infty} \gamma^{k+1} \mathbb{E}\Big[\Big| w_k^{\rm{T}} \big(  \tilde{P} -P \big)\tilde{A}_K^{k+1} x \nonumber \\
    &\quad + w_k^{\rm{T}} P  \big(\tilde{A}_K^{k+1} - A_K^{k+1} \big)x \Big|\Big] \nonumber\\
    &\leq 2 \sum_{k = 0 }^{\infty} \gamma^{k+1} \mathbb{E} \Big[ \left\| w_k \right\| \left\| \tilde{P} -P \right\| \left\|\tilde{A}_K^{k+1}  \right\| \left\| x\right\| \Big]\nonumber \\
    &\quad + 2 \sum_{k = 0 }^{\infty} \gamma^{k+1} \mathbb{E} \Big[\left\| w_k \right\| \left\| P\right\|  \left\| \tilde{A}_K^{k+1}-A_K^{k+1}\right\|  \left\| x\right\|\Big]\nonumber \\
    &\leq 2 \sigma \left\| \tilde{P} -P\right\| \left\|x\right\| \sum_{k = 0 }^{\infty} \gamma^{k+1} \left\| \tilde{A}_K^{k+1}\right\|  \nonumber \\
    & \quad + 2 \sigma \left\| P \right\| \left\| x\right\| \sum_{k = 0 }^{\infty} \gamma^{k+1} \left\|\tilde{A}_K^{k+1} -  A_K^{k+1} \right\|.
\end{align}
By decomposing the term $\tilde{A}_K^{k+1} -  A_K^{k+1}= (\tilde{A}_K - A_K) ( \sum_{i=0}^k \tilde{A}_K^{k-i} A_K^{i} )$, we have 
\begin{align}\label{eq:pf:pert:6}
    &\left\| \tilde{A}_K^{k+1} -  A_K^{k+1}\right\| \nonumber \\
    &\leq    \left\|\tilde{A}_K - A_K \right\|  \left\|\sum_{i=0}^k \tilde{A}_K^{k-i} A_K^{i} \right\| \nonumber \\
    &\leq  \left\|\tilde{A}_K - A_K \right\| ( \sum_{i=0}^k \left\| \tilde{A}_K^{k-i} A_K^{i} \right\|) \nonumber \\
    &\leq \left\|\Delta A_K \right\| (\sum_{i=0}^k  \left\| \tilde{A}_K^{k-i} \right\| \left\|A_K^{i} \right\|  ) \nonumber \\
    &\leq  \left\|\Delta A_K \right\| (k+1) \rho_K^k \leq  \left\|\Delta A_K \right\| U.
\end{align}
Define the function  $ f(k)=(k+1)\rho_K^k$, and the constant $U = \max\{1, \frac{1}{\rho_0} \rho_K^{(1-\rho_0)/\rho_0} \}$, $\rho_0 = \ln(1/\rho_K)$.
It is easy to verify that the function $ f(k)$ obtains the maximum at $k=0$ if $\rho_0\geq1$ and at $k=\frac{1-\rho_0}{\rho_0}$ if $\rho_0<1$. Therefore, the last inequality follows since $f(k)\leq  \max\{1, \frac{1}{\rho_0} \rho_K^{(1-\rho_0)/\rho_0} \} = U$ for all $k\geq0$.
Substituting \eqref{eq:pf:pert:6} into \eqref{eq:pf:pert:5}, we have
\begin{align}\label{eq:pf:pert:7}
    &2 \sum_{k = 0 }^{\infty} \gamma^{k+1} \mathbb{E}\Big[\Big| w_k^{\rm{T}} (  \tilde{P} \tilde{A}_K^{k+1} -P A_K^{k+1} )x \Big| \Big]\nonumber \\
    &\leq  \frac{2 \sigma \left\| P \right\| \left\| x\right\| U\gamma}{1-\gamma}   \left\|\Delta A_K \right\| \nonumber \\
    & \quad + 2 \sigma \left\| \tilde{P} -P\right\| \left\| x\right\| \sum_{k = 0 }^{\infty} \gamma^{k+1} \left\| \tilde{A}_K^{k+1}\right\| \nonumber \\
    & \leq  \frac{2 \sigma \left\| P \right\| \left\| x\right\| U\gamma}{1-\gamma}   \left\|\Delta A_K \right\| +  2 \sigma \left\| \tilde{P} -P\right\| \left\| x\right\| \sum_{k = 0 }^{\infty} (\gamma \rho_K)^{k+1} \nonumber \\
    &\leq \frac{2 \sigma \left\| P \right\| \left\| x\right\| U\gamma}{1-\gamma}   \left\|\Delta A_K \right\| + \frac{2 \sigma \gamma \rho_K} {1-\gamma \rho_K}  \left\| \tilde{P} -P\right\|,
\end{align}
where the second inequality follows since $\left\| \tilde{A}_K^{k+1} \right\| \leq\left\| \tilde{A}_K \right\|^{k+1} \leq \rho_K^{k+1} $.
For the fourth term, we have
\begin{align}\label{eq:pf:pert:8}
    &2  \sum_{k = 1 }^{\infty} \gamma^{k+1} \mathbb{E}\Big[\Big| w_k^{\rm{T}} \big( \tilde{P} \sum_{\tau=0}^{k-1} \tilde{A}_K^{k-\tau}w_{\tau} - P \sum_{\tau=0}^{k-1} {A}_K^{k-\tau}w_{\tau}   \big) \Big| \Big]\nonumber \\
    & = 2  \sum_{k = 1 }^{\infty} \gamma^{k+1} \mathbb{E}\Big[\Big| w_k^{\rm{T}} \Big( \tilde{P} \sum_{\tau=0}^{k-1} \tilde{A}_K^{k-\tau}w_{\tau} - P \sum_{\tau=0}^{k-1} \tilde{A}_K^{k-\tau}w_{\tau}     \nonumber \\
    &\quad +  P \sum_{\tau=0}^{k-1} \tilde{A}_K^{k-\tau}w_{\tau} - P \sum_{\tau=0}^{k-1} {A}_K^{k-\tau}w_{\tau}   \Big) \Big| \Big]\nonumber \\
    &\leq 2  \sum_{k = 1 }^{\infty} \gamma^{k+1} \mathbb{E}\Big[\Big|  w_k^{\rm{T}} \big( \tilde{P}-P \big) \sum_{\tau=0}^{k-1} \tilde{A}_K^{k-\tau}w_{\tau} \Big| \Big]\nonumber \\
    &\quad + 2 \sum_{k = 1 }^{\infty} \gamma^{k+1} \mathbb{E}\Big[\Big| w_k^{\rm{T}}  P \sum_{\tau=0}^{k-1} \big(\tilde{A}_K^{k-\tau} - {A}_K^{k-\tau} \big)w_{\tau}  \Big| \Big] \nonumber \\
    &\leq 2  \sum_{k = 1 }^{\infty} \gamma^{k+1} \mathbb{E} \Big[  \left\|w_k \right\|  \left\|\tilde{P}-P \right\| \left\| \sum_{\tau=0}^{k-1} \tilde{A}_K^{k-\tau}w_{\tau} \right\|  \Big]   \nonumber \\
    &\quad + 2 \sum_{k = 1 }^{\infty} \gamma^{k+1} \mathbb{E}\Big[ \left\|w_k \right\| \left\|P \right\|  \left\|  \sum_{\tau=0}^{k-1} \big(\tilde{A}_K^{k-\tau} - {A}_K^{k-\tau} \big)w_{\tau}  \right\| \Big]\nonumber \\
    &\leq 2 \sum_{k = 1 }^{\infty} \gamma^{k+1} \mathbb{E}\Big[ \left\|w_k \right\| \left\|{P} \right\|  \sum_{\tau=0}^{k-1} \left\| \tilde{A}_K^{k-\tau} - {A}_K^{k-\tau} \right\|   \left\|w_{\tau}  \right\| \Big] \nonumber \\
    &\quad + 2  \sum_{k = 1 }^{\infty} \gamma^{k+1} \mathbb{E} \Big[  \left\|w_k \right\|  \left\|\tilde{P}-P \right\|  \sum_{\tau=0}^{k-1} \left\|\tilde{A}_K^{k-\tau}\right\| \left\|w_{\tau} \right\| \Big] \nonumber \\
    &\leq 2 \sigma^2 \left\|{P} \right\| \sum_{k = 1 }^{\infty} \gamma^{k+1} \sum_{\tau=0}^{k-1} \left\| \tilde{A}_K^{k-\tau} - {A}_K^{k-\tau} \right\| \nonumber \\
    &\quad + 2  \sigma^2 \left\|\tilde{P}-P \right\|  \sum_{k = 1 }^{\infty} \gamma^{k+1} \sum_{\tau=0}^{k-1} \left\|\tilde{A}_K^{k-\tau}\right\|,
\end{align}
where the last inequality follows since $w_k$ and $w_{\tau}$, $\tau=0,1,\ldots,k-1$ are independent. 
By decomposing the term $\tilde{A}_K^{k-\tau} - {A}_K^{k-\tau}$, we have
\begin{align}\label{eq:pf:pert:9}
    &\sum_{\tau=0}^{k-1} \left\| \tilde{A}_K^{k-\tau} - {A}_K^{k-\tau} \right\| \nonumber \\
    &=\sum_{\tau=0}^{k-1} \left\| (\tilde{A}_K -{A}_K )  \sum_{i=0}^{k-\tau-1} \tilde{A}_K^{k-\tau-1-i} A_K^{i}  \right\| \nonumber\\
    &\leq  \sum_{\tau=0}^{k-1} \left\|  \tilde{A}_K -{A}_K\right\| \left\| \sum_{i=0}^{k-\tau-1} \tilde{A}_K^{k-\tau-1-i} A_K^{i}  \right\| \nonumber \\
    &\leq  \left\|  \tilde{A}_K -{A}_K\right\| \sum_{\tau=0}^{k-1} \sum_{i=0}^{k-\tau-1} \rho_K^{k-\tau -1} \nonumber \\
    & = \left\|  \Delta{A}_K\right\| \sum_{\tau=0}^{k-1}  (k-\tau)\rho_K^{k-\tau -1} \nonumber \\
    & = \left\|  \Delta{A}_K\right\| \sum_{i=1}^k i \rho_K^{i-1} \nonumber \\
    &\leq  \frac{\left\|  \Delta{A}_K\right\|}{(1-\rho_K)^2},
\end{align}
where the last inequality follows since $S_k :=\sum_{i=1}^k i \rho_K^{i-1} = \frac{1-\rho_K^k}{(1-\rho_K)^2} - \frac{k\rho_K^k}{1-\rho_K} \leq \frac{1}{(1-\rho_K)^2}$.
Substituting \eqref{eq:pf:pert:9} into \eqref{eq:pf:pert:8}, we have
\begin{align}\label{eq:pf:pert:10}
    &2  \sum_{k = 1 }^{\infty} \gamma^{k+1} \mathbb{E}\Big[\Big| w_k^{\rm{T}} \big( \tilde{P} \sum_{\tau=0}^{k-1} \tilde{A}_K^{k-\tau}w_{\tau} - P \sum_{\tau=0}^{k-1} {A}_K^{k-\tau}w_{\tau}   \big) \Big| \Big]\nonumber \\
    & \leq 2 \sigma^2 \left\|{P} \right\| \sum_{k = 1 }^{\infty} \gamma^{k+1} \frac{\left\|  \Delta{A}_K\right\|}{(1-\rho_K)^2} \nonumber \\
    &\quad + 2  \sigma^2 \left\|\tilde{P}-P \right\|  \sum_{k = 1 }^{\infty} \gamma^{k+1} \sum_{\tau=0}^{k-1} \left\|\tilde{A}_K^{k-\tau}\right\| \nonumber \\
    &\leq \frac{2 \sigma^2 \left\|{P} \right\| \gamma^2 }{(1-\gamma)(1-\rho_K)^2} \left\|\Delta A_K \right\| \nonumber \\
    &\quad + 2  \sigma^2 \left\|\tilde{P}-P \right\|  \sum_{k = 1 }^{\infty} \gamma^{k+1}\sum_{\tau=0}^{k-1} \rho^{k-\tau} \nonumber \\
    &\leq \frac{2 \sigma^2 \left\|{P} \right\| \gamma^2 \left\|\Delta A_K \right\|}{(1-\gamma)(1-\rho_K)^2}  + \frac{2  \sigma^2 \left\|\tilde{P}-P \right\|}{(1-\gamma)(1-\rho_K)}. 
\end{align}
Combining \eqref{eq:pf:pert:3}, \eqref{eq:pf:pert:4}, \eqref{eq:pf:pert:7}, \eqref{eq:pf:pert:10} and \eqref{eq:pf:pert:2}, we have
\begin{align}\label{eq:pf:pert:11}
    &\mathbb{E}\big[|\tilde{Y}|\big] \leq \big(\left\|x\right\|^2  + \frac{\sigma^2 \gamma }{1-\gamma} \big)\left\|  \tilde{P} - P\right\| \nonumber \\
    & \quad + \frac{2 \sigma \left\| {P} \right\| \left\| x\right\| U\gamma}{1-\gamma}   \left\|\Delta A_K \right\| + \frac{2 \sigma \gamma \rho_K} {1-\gamma \rho_K}  \left\| \tilde{P} -P\right\| \nonumber \\
    & \quad + \frac{2 \sigma^2 \left\|{P} \right\| \gamma^2 \left\|\Delta A_K \right\|}{(1-\gamma)(1-\rho_K)^2}  + \frac{2  \sigma^2 \left\|\tilde{P}-P \right\|}{(1-\gamma)(1-\rho_K)} \nonumber \\
      &:= \tilde{c}_3 \left\|\tilde{P}-P \right\| + \tilde{c}_4 \left\|\Delta A_K \right\|,
\end{align}
where 
\begin{align*}
    &\tilde{c}_3=\left\|x\right\|^2  + \frac{\sigma^2 \gamma }{1-\gamma} +\frac{2 \sigma \gamma \rho_K} {1-\gamma \rho_K} + \frac{2  \sigma^2}{(1-\gamma)(1-\rho_K)}, \\
&\tilde{c}_4=\frac{2 \sigma \left\| {P} \right\| \left\| x\right\| U\gamma}{1-\gamma}  + \frac{2 \sigma^2 \left\|{P} \right\| \gamma^2}{(1-\gamma)(1-\rho_K)^2}.
\end{align*}

Substituting  \eqref{eq:pf:pert:11} into \eqref{eq:pf:pert:1} and using \eqref{eq:perturb:bound:P}, we have 
\begin{align*}
    &\sup_{z}|{F}^{K}_x(z)-\tilde{F}^{K}_{x}(z) | \nonumber \\
    &\leq \tilde{f}_{\max}(\tilde{c}_3 \left\|\tilde{P}-P \right\| + \tilde{c}_4 \left\|\Delta A_K \right\|) \nonumber \\
    &\leq \tilde{f}_{\max} \tilde{c}_3  \frac{2 \left\| P \right\|_F }{l - 2 \epsilon} \big( \gamma \left\| A_K \right\|_F \left\| \Delta A_K \right\|_F + \frac{\gamma}{2} \left\| \Delta A_K \right\|_F^2\big) \nonumber \\
    &\quad + \tilde{f}_{\max} \tilde{c}_4  \left\|\Delta A_K \right\| \nonumber \\
    & \leq \Big(\frac{2 \tilde{f}_{\max} \tilde{c}_3 \sqrt{n}  \gamma \left\|A_K\right\|_F \left\| P \right\|_F } {l - 2 \epsilon} +\tilde{f}_{\max} \tilde{c}_4 \Big) \left\|\Delta A_K \right\| \nonumber \\
    &\quad + \frac{\tilde{f}_{\max} \tilde{c}_3 n \gamma \left\| P \right\|_F }{l - 2 \epsilon} \left\| \Delta A_K \right\|^2 \nonumber \\
    & := \tilde{c}_1 \left\|\Delta A_K \right\| + \tilde{c}_2 \left\| \Delta A_K \right\|^2,
\end{align*}
where the last inequality follows from $\left\|M\right\|_F \leq \sqrt{n}\left\|M\right\|_2$ for any matrix $M \in \mathbb{R}^{n\times n}$.
The proof is complete and also yields the expression of the constants $\tilde{c}_1,\tilde{c}_2$.

\subsection{Proof of Theorem \ref{The:ranretromodelapprox}}\label{app:thm:modelappro}
Define $Y_N:= G^K(x) - G^K_N(x)$. Then, we have
\begin{align}\label{eq:upp_bound_temp1}
     &\sup_{z}|{F}^{K}_x(z)-{F}^{K}_{x,N}(z) | \nonumber \\
     &=\sup_{z}|\mathbb{P}(G^K_N(x) \leq z) -\mathbb{P}(G^K(x)\leq z) | \nonumber \\
     &= \sup_{z}|\mathbb{P}(G^K_N(x)\leq z) -\mathbb{P}(G^K_N(x) +Y_N\leq z) | \nonumber \\
     &= \sup_{z}\Big|\mathbb{P}(G^K_N(x)\leq z) \int_{-\infty}^{\infty} \mathbb{P}(Y_N = t)dt \nonumber \\
     &\quad -\int_{-\infty}^{\infty} \mathbb{P}(G^K_N(x) \leq z-t) \mathbb{P}(Y_N = t) dt \Big| \nonumber \\
     &= \sup_{z}\Big| \int_{-\infty}^{\infty} \mathbb{P}(Y_N = t) \big(  F^K_{x,N}(z) -F^K_{x,N}(z-t) \big) dt \Big|.
\end{align}
Since the random variables $w_t$ are i.i.d for all $t>0$ and the PDF of $w_t$ exists, the function $F^K_{x,N}$ is continuous and differentiable. 
Applying the mean value theorem, when $t>0$ there exists a point $z'\in[z-t,z]$ such that $F^K_{x,N}(z) -F^K_{x,N}(z-t) = f^K_{x,N}(z') t$, where $f^K_{x,N}$ is the PDF of $G^K_N(x)$. Since the PDF of $w_t$ is bounded, it further follows  that $f^K_{x,N}$ is bounded.   Then, we have that $|F^K_{x,N}(z) -F^K_{x,N}(z-t)| = |f^K_{x,N}(z') t| \leq f_{\max} |t|$, where $f_{\max}$ is an upper bound of the probability function $f^K_{x,N}$. Following a similar argument, we can show that this inequality holds when $t\leq 0$. Substituting this inequality into \eqref{eq:upp_bound_temp1}, we obtain
\begin{align}\label{eq:upp_bound_temp2}
   &\sup_{z}|{F}^{K}_x(z)-{F}^{K}_{x,N}(z) | \nonumber \\
   \leq &\sup_{z}\Big| \int_{-\infty}^{\infty} \mathbb{P}(Y_N = t) f_{\max} |t| dt \Big| 
     = f_{\max} \mathbb{E}|Y_N|.
\end{align}
From the definition of $Y_N$, we obtain that
\begin{align*}
     Y_N =& 
     \sum_{k = N }^{\infty}  \gamma^{k+1}  w_k^{\rm{T}} P w_k +  2 \sum_{k = N }^{\infty} \gamma^{k+1} w_k^{\rm{T}} P A_K^{k+1}x \\
     &+ 2 \sum_{k = N }^{\infty} \gamma^{k+1} w_k^{\rm{T}} P \sum_{\tau=0}^{k-1} A_K^{k-\tau}w_{\tau} \nonumber \\
     \mathop{=}^{t=k-N}&    \gamma^N \Big(  \sum_{t = 0 }^{\infty} \gamma^{t+1}  w_{t+N}^{\rm{T}} P w_{t+N} \nonumber \\
     &+  2 \sum_{t = 0 }^{\infty} \gamma^{t+1} w_{t+N}^{\rm{T}} P A_K^{t+N+1}x \nonumber \\
     &+ 2 \sum_{t = 0 }^{\infty} \gamma^{t+1} w_{t+N}^{\rm{T}} P \sum_{\tau=0}^{t+N-1} A_K^{t+N-\tau}w_{\tau}  \Big).
\end{align*}
Taking the expectation of the absolute value of $Y_N$, we have
\begin{align*}
    \mathbb{E}\big[|Y_N|\big] \leq &\gamma^N \Big( \sum_{t = 0 }^{\infty} \gamma^{t+1} \mathbb{E} \big[| w_{t+N}^{\rm{T}} P w_{t+N} | \big] \\
    &+  2 \sum_{t = 0 }^{\infty} \gamma^{t+1} \mathbb{E}\big[ | w_{t+N}^{\rm{T}} P A_K^{t+N+1}x|\big]\nonumber \\
    &+ 2 \sum_{t = 0 }^{\infty} \gamma^{t+1} \mathbb{E} \big[|w_{t+N}^{\rm{T}} P \sum_{\tau=0}^{t+N-1} A_K^{t+N-\tau}w_{\tau} | \big] \Big) .
\end{align*}
We handle the terms in the above inequality one by one. For the first term, we have that
\begin{align}\label{eq:upp_bound_temp4}
    &\sum_{t = 0 }^{\infty} \gamma^{t+1} \mathbb{E} \big[| w_{t+N}^{\rm{T}} P w_{t+N} | \big]\nonumber \\
    &\leq \sum_{t = 0 }^{\infty} \gamma^{t+1}  \mathbb{E}\big[| \lambda_{\max}(P) w_{t+N}^{\rm{T}} w_{t+N} | \big]\nonumber \\
    &\leq \lambda_{\max}(P) \sigma^2 \frac{\gamma}{1-\gamma}.
\end{align}
By virtue of Jensen's Inequality, it gives $\mathbb{E}^2 [\left\| w_k \right\|] \leq \mathbb{E}[\left\| w_k \right\|^2 ] \leq \sigma^2 $ . Then, for the second term, we have
\begin{align}\label{eq:upp_bound_temp5}
    &2 \sum_{t = 0 }^{\infty} \gamma^{t+1} \mathbb{E}\big[| w_{t+N}^{\rm{T}} P A_K^{t+N+1}x| \big]\nonumber \\
    \leq & 2\sigma  \sum_{t = 0 }^{\infty} \gamma^{t+1} \left\| P \right\| \left\| A_K^{t+N+1} \right\| \left\|x \right\| \nonumber \\
    \leq& 2\sigma  \sum_{t = 0 }^{\infty} \gamma^{t+1} \left\| P\right\| \rho_K^{t+N-1} \left\|x \right\| \nonumber \\
    \leq & 
    2\sigma \left\| P\right\| |x|\frac{\gamma \rho_K^{N-1}}{1-\gamma \rho_K} \leq 2\sigma \left\| P\right\| \left\|x\right\| \frac{\gamma}{1-\gamma \rho_K},
\end{align}
where the second inequality is due to the fact that $\left\|A_K^{t+N+1} \right\|\leq (\left\|A_K \right\|)^{t+N+1} \leq \rho_K^{t+N+1}$ and the last inequality follows from the fact that $N\geq 1$.
For the third term, we have that
\begin{align}\label{eq:upp_bound_temp8}
    &2 \sum_{t = 0 }^{\infty} \gamma^{t+1} \mathbb{E} \big[|w_{t+N}^{\rm{T}} P \sum_{\tau=0}^{t+N-1} A_K^{t+N-\tau}w_{\tau} | \big]\nonumber \\
    \leq  & 2 \sum_{t = 0 }^{\infty} \gamma^{t+1} \mathbb{E} \left[ \left\| w_{t+N}^{\rm{T}}\right\|  \left\| P \right\|  \left\|\sum_{\tau=0}^{t+N-1} A_K^{t+N-\tau} w_{\tau}\right\| \right] \nonumber \\
    \leq &  2 \sigma \left\| P \right\|  \sum_{t = 0 }^{\infty} \gamma^{t+1}   \mathbb{E} \left[ \left\|\sum_{\tau=0}^{t+N-1} A_K^{t+N-\tau} w_{\tau}\right\| \right] \nonumber \\
    \leq  & 2 \sigma \left\| P \right\|  \sum_{t = 0 }^{\infty} \gamma^{t+1}   \mathbb{E} \left[ \sum_{\tau=0}^{t+N-1} \left\|A_K^{t+N-\tau} \right\|\left\|w_{\tau}\right\| \right] \nonumber\\
    \leq  &2 \sigma^2 \left\| P \right\| \sum_{t = 0 }^{\infty} \gamma^{t+1}  \sum_{\tau=0}^{t+N-1}  \rho_K^{t+N-\tau} \nonumber \\
    \leq & 2 \sigma^2 \left\| P \right\| \sum_{t = 0 }^{\infty} \gamma^{t+1} \frac{\rho_K}{1-\rho_K} \nonumber \\
    \leq & 2 \sigma^2 \left\| P \right\| \frac{\gamma \rho_K}{(1-\gamma)(1-\rho_K)},
\end{align}
where the second inequality is due to the fact that $w_{\tau}$ and $w_{t+N}$ are independent and the  second to last inequality follows from the fact that $\sum_{\tau=0}^{t+N-1}  \rho_K^{t+N-\tau} = \sum_{\tau=1}^{t+N} \rho_K^{\tau} \leq \frac{\rho_K}{1-\rho_K} $.
Combining \eqref{eq:upp_bound_temp4}, \eqref{eq:upp_bound_temp5} and \eqref{eq:upp_bound_temp8}, we have that
\begin{align*}
    &\sup_{z}|{F}^{K}_x(z)-{F}^{K}_{x,N}(z) | \leq  f_{\max} \mathbb{E}\big[|Y_N|\big] \nonumber \\
    \leq & f_{\max} \gamma^N \Big(\lambda_{\max}(P) \sigma^2 \frac{\gamma}{1-\gamma} + 2\sigma \left\| P\right\| \left\|x\right\| \frac{\gamma}{1-\gamma \rho_K} \nonumber \\
    &+  2 \sigma^2 \left\| P \right\| \frac{\gamma \rho_K}{(1-\gamma)(1-\rho_K)} \Big) 
    :=  c_0 \gamma^N . \nonumber 
\end{align*}
The proof is complete and also yields the expression of the constant $c_0$.

\subsection{Proof of Theorem \ref{The:ranretromodelfreeapprox}}\label{app:thm:modelfreeappro}

It follows that 
\begin{align}\label{eq:thm:r1}
    &\sup_z | \hat{F}^{K,T}_{x,M}(z) - F^K_x(z)| \nonumber \\
    \leq & \sup_z | \hat{F}^{K,T}_{x,M}(z) - F^{K,T}_x(z)| + \sup_z| F^{K,T}_x(z) - F^K_x(z)|. 
\end{align}
Note that $\hat{F}^{K,T}_{x,M}$ and  $F^{K,T}_x$ are the EDF and CDF of the random variable $G^{K,T}(x)$, respectively. By virtue of the Dvoretzky–Kiefer–Wolfowitz inequality, we have
\begin{align}\label{eq:thm:r2}
    \sup_z | \hat{F}^{K,T}_M(z) - F^{K,T}(z)| \leq \sqrt{\frac{\ln(1 / \delta)}{2M}},
\end{align}
with probability at least $1-\delta$.
Define the random variable $Z_T = G^K(x) - G^{K,T}(x) = \sum_{T+1}^{\infty} \gamma^t x_t^{\rm{T}} (Q+K^{\rm{T}}RK)x_t$
Further, we have
\begin{align}\label{eq:thm:r3}
    &\sup_z| F^{K,T}_x(z) - F^K_x(z)| \nonumber \\
    &=\sup_z\Big| \mathbb{P}\{G^{K,T}(x)\leq z \} -\mathbb{P}\{G^{K}(x)\leq z \}\Big| \nonumber \\
    & = \sup_z\Big| \mathbb{P}\{G^{K,T}(x)\leq z \} -\mathbb{P}\{G^{K,T}(x) +Z_T\leq z \}\Big| \nonumber \\
    &= \sup_{z}\Big|\mathbb{P}(G^{K,T}(x)\leq z) \int_{-\infty}^{\infty} \mathbb{P}(Z_T = t)dt \nonumber \\
     &\quad -\int_{-\infty}^{\infty} \mathbb{P}(G^{K,T}(x) \leq z-t) \mathbb{P}(Z_T = t) dt \Big| \nonumber \\
     &= \sup_{z}\Big| \int_{-\infty}^{\infty} \mathbb{P}(Z_T = t) \big(  F^{K,T}_{x}(z) -F^{K,T}_{x}(z-t) \big) dt \Big| \nonumber \\
     &\leq \sup_{z}\Big| \int_{-\infty}^{\infty} \mathbb{P}(Z_T = t) f_{\max} |t| dt \Big| \nonumber \\
     &= f_{\max} \mathbb{E}\big[|Z_T|\big].
\end{align}
Now we focus on $\mathbb{E}|Z_T|$. From its definition, it gives
\begin{align}\label{eq:thm:t0}
    &\mathbb{E}\big[|Z_T|\big]= \mathbb{E}\Big[ \sum_{t=T+1}^{\infty} \gamma^t x_t^{\rm{T}} (Q+K^{\rm{T}} R K) x_t \Big] \nonumber \\
    & \leq \mathbb{E}\Big[ \sum_{t=T+1}^{\infty} \gamma^t \left\|Q + K^{\rm{T}} R K \right\| \left\| x_t\right\|^2\Big] .
\end{align}
From $x_{t+1} = A_K x_t + w_t$, where $A_K = A+BK$, we have
\begin{align}
    x_t = A_K^t x + \sum_{\tau=0}^{t-1} A_K^{t-1-\tau} w_{\tau} .
\end{align}
Hence,
\begin{align}\label{eq:thm:t1}
    \mathbb{E}[ \left\| x_t\right\|^2] 
    & = \mathbb{E}\Big[ \left\|A_K^t x + \sum_{\tau=0}^{t-1} A_K^{t-1-\tau} w_{\tau}  \right\|^2\Big]\nonumber \\
    &\leq \mathbb{E}\Big[ \left\|A_K^t x\right\|^2 +  \left\|\sum_{\tau=0}^{t-1} A_K^{t-1-\tau} w_{\tau}  \right\|^2 \nonumber \\
    &\quad + 2 \left\|A_K^t x\right\|  \left\|\sum_{\tau=0}^{t-1} A_K^{t-1-\tau} w_{\tau}  \right\| \Big] \nonumber \\
    &\leq \rho_K^{2t} \left\| x\right\|^2 +\mathbb{E}\Big[ \left\|\sum_{\tau=0}^{t-1} A_K^{t-1-\tau} w_{\tau}  \right\|^2 \nonumber \\
    &\quad + 2\rho_K^t \left\| x\right\|\left\|\sum_{\tau=0}^{t-1} A_K^{t-1-\tau} w_{\tau}  \right\| \Big],
\end{align}
where the last inequality follows from $\left\| A_K^t x\right\| \leq \left\| A_K^t \right\| \left\| x \right\| \leq \rho_K^t \left\| x\right\|$.
Further, we have
\begin{align}\label{eq:thm:t1.1}
    &\mathbb{E}\left[\left\| \sum_{\tau=0}^{t-1} A_K^{t-1-\tau} w_{\tau}  \right\|\right] \leq \mathbb{E} \left[ \sum_{\tau=0}^{t-1} \left\| A_K^{t-1-\tau}\right\| \left\|w_{\tau}  \right\| \right] \nonumber \\
    &\leq \sigma \sum_{\tau=0}^{t-1} \left\| A_K^{t-1-\tau}\right\| \leq \sigma \sum_{\tau=0}^{t-1} \rho_K^{t-1-\tau} \leq \frac{\sigma}{1-\rho_K},
\end{align}
where the first inequality follows from the Cauchy–Schwarz inequality,
the second inequality follows from  $\mathbb{E}^2 [\left\| w_k \right\|] \leq \mathbb{E}[\left\| w_k \right\|^2 ] \leq \sigma^2 $, and the third inequality follows from $\left\|A_K^{t+N+1} \right\|\leq (\left\|A_K \right\|)^{t+N+1} \leq \rho_K^{t+N+1}$.
Further,
\begin{align}\label{eq:thm:t1.2}
    &\mathbb{E}\left[\left\| \sum_{\tau=0}^{t-1} A_K^{t-1-\tau} w_{\tau}  \right\|^2\right]\nonumber \\
    &=\mathbb{E}\left[ \sum_{\tau=0}^{t-1} w_{\tau}^{\rm{T}} (A_K^{t-1-\tau})^{\rm{T}} A_K^{t-1-\tau} w_{\tau}  \right]\nonumber \\
    &\leq \mathbb{E}\left[ \sum_{\tau=0}^{t-1} \left\| (A_K^{t-1-\tau})^{\rm{T}} A_K^{t-1-\tau} \right\| \left\| w_{\tau}\right\|^2 \right] \nonumber \\
    & \leq \sigma^2 \sum_{\tau=0}^{t-1} \left\| (A_K^{t-1-\tau})^{\rm{T}} A_K^{t-1-\tau} \right\| \nonumber \\
    &\leq \sigma^2 \sum_{\tau=0}^{t-1} \rho_K^{2(t-1-\tau)} \leq \frac{\sigma^2}{1-\rho_K^2},
\end{align}
where the first equality follows from the fact that the random variables $w_{\tau}$, $\tau\in \mathbb{N}$, are i.i.d. and with zero mean. The first inequality follows from the Cauchy-Schwarz inequality and the third inequality follows from $\left\| (A_K^{t-1-\tau})^{\rm{T}} A_K^{t-1-\tau} \right\| \leq \left\| (A_K^{t-1-\tau})^{\rm{T}}\right\| \left\|A_K^{t-1-\tau} \right\| \leq \rho_K^{2(t-1-\tau)} $.
Substituting \eqref{eq:thm:t1.1} and \eqref{eq:thm:t1.2} into \eqref{eq:thm:t1}, we have
\begin{align}\label{eq:thm:t2}
    &\mathbb{E}[ \left\| x_t\right\|^2] \leq \rho_K^{2t} \left\| x\right\|^2 + \frac{\sigma^2}{1-\rho_K^2} +  \frac{2\rho_K^t \left\| x\right\|\sigma}{1-\rho_K}.
\end{align}
Substituting \eqref{eq:thm:t2} into \eqref{eq:thm:t0}, we have
\begin{align}\label{eq:thm:r4}
    &\mathbb{E}\big[|Z_T|\big] \nonumber \\
    &\leq \left\|Q + K^{\rm{T}} R K \right\| \sum_{t=T+1}^{\infty} \gamma^t \Big(  \rho_K^{2t} \left\| x\right\|^2 + \frac{\sigma^2}{1-\rho_K^2} \nonumber \\
    &\quad +  \frac{2\rho_K^t \left\| x\right\|\sigma}{1-\rho_K} \Big) \nonumber \\
    &\leq \left\|Q + K^{\rm{T}} R K \right\| \Big( \frac{\gamma^{T+1} \rho_K^{2(T+1)} \left\| x\right\|^2}{1-\gamma \rho_K^2} \nonumber \\
    &\quad +  \frac{2 \gamma^{T+1} \rho_K^{T+1} \left\|x\right\| \sigma }{(1-\rho_K)(1-\gamma \rho_K)}  +  \frac{\gamma^{T+1} \sigma^2}{(1-\gamma)(1-\rho_K^2)}\Big) \nonumber \\
    &  = \left\| Q_K\right\| \gamma^{T+1}(c_1 \rho_K^{2(T+1)}  + c_2 \rho_K^{T+1}  + c_3).
\end{align}
Combining \eqref{eq:thm:r1}, \eqref{eq:thm:r2}, \eqref{eq:thm:r3} and \eqref{eq:thm:r4}, it gives
\begin{align*}
    &\sup_z | \hat{F}^{K,T}_M(z) - F^K(z)| \nonumber \\
    \leq & \sup_z | \hat{F}^{K,T}_M(z) - F^{K,T}(z)| + \sup_z| F^{K,T}(z) - F^K(z)|\nonumber \\
    \leq &  f_{\max} \left\|Q_K \right\| \gamma^{T+1}(c_1 \rho_K^{2(T+1)}  + c_2 \rho_K^{T+1}  + c_3) +\sqrt{\frac{\ln(1 /\delta)}{2M}} ,
\end{align*}
which completes the proof.



\bibliography{bibfile}
\bibliographystyle{unsrt}

\end{document}